\documentclass [12pt,a4paper]{amsart}
\usepackage{textcomp}
\usepackage{amsmath}
\usepackage{cases}
\usepackage{amscd}
\usepackage{epsfig}
\usepackage{graphicx}
\usepackage{verbatim}
\usepackage{amssymb}
\usepackage{amsfonts}
\usepackage{amsbsy}
\usepackage{amsthm}
\usepackage{amstext}
\usepackage{amsopn}
\usepackage[dvips]{color}
\usepackage{indentfirst}
\usepackage{mathrsfs}

\newtheorem{thm}{Theorem}[section]
\newtheorem{defi}[thm]{Definition}
\newtheorem{cor}[thm]{Corollary}
\newtheorem{lem}[thm]{Lemma}
\newtheorem{prop}[thm]{Proposition}
\newtheorem{rem}[thm]{Remark}

\newtheorem{eg}[thm]{Example}

\def\<{\langle}    \def\>{\rangle}
\def\lp{\left(}  \def\rp{\right)}

\newcommand{\ve}{(V,E)}
\newcommand{\g}{(V, w,\mu)}
\newcommand{\go}{(V_o, w_o,\mu_o)}

\newcommand{\cl}{\text{\rm{cl}} }
\newcommand{\inte}{\text{\rm{int}} }
\newcommand{\ccv}{C_c(V) }
\newcommand{\ltwo}{L^2(V, \mu) }

\newcommand{\ddt}{\frac{\partial}{\partial t} }
\newcommand{\lap}{\frac{1}{\mu(x)}\sum_{y\in V}w(x,y) }
\newcommand{\xt}{\lp X_t\rp_{t\geq0} }

\def\PP{\mathbb{P}}
\def\ZZ{\mathbb{Z}}
\def\EE{\mathbb{E}}
\def\NN{\mathbb{N}}
\def\RR{\mathbb{R}}

\def\aa{\mathcal{A}}
\def\ll{\mathcal{L}}

\def\ff{\mathcal{F}}
\def\nn{\mathcal{N}}

\def\uu{\mathcal{U}}
\def\vv{\mathcal{V}}
\def\tt{\mathcal{T}}
\def\mm{\mathcal{M}}

\def\ddd{\mathscr{D}}
\def\lll{\mathscr{L}}
\def\eee{\mathscr{E}}
\def\fff{\mathscr{F}}

\def\ppp{\mathscr{P}}

\def\e{\varepsilon}
\def\p{\varphi}
\def\a{\alpha}
\def\s{\sigma}

\author[X. Huang]{Xueping Huang}
\address{Faculty of Mathematics and Information\\ FSU Jena, 07743 Jena, Germany}
\thanks{Research supported partially by Project CRC701 and DFG}
\email{xueping.huang@uni-jena.de}
\author[Y. Shiozawa]{Yuichi Shiozawa}
\address{Graduate School of Natural Science and Technology,
Department of Environmental and Mathematical Sciences, \\
Okayama University\\3-1-1, Tsushima-naka, Okayama, 700-8530,
Japan}
\email{shiozawa@ems.okayama-u.ac.jp }

\title[Upper escape rate]{Upper escape rate of Markov chains on weighted graphs}
\def\func#1{\mathop{\rm #1}\nolimits}%
\def\FRAME#1#2#3#4#5#6#7#8
{
 \begin{figure}[ptbh]
 \begin{center}
 \includegraphics[height=#3]{#7}
 \caption{#5}
 \label{#6}
 \end{center}
 \end{figure}
}

\begin{document}
\numberwithin{equation}{section}

\begin{abstract}
We obtain
an upper escape rate function for a continuous time minimal symmetric Markov chain, defined on a locally finite weighted graph.
This upper rate function is given in terms of volume growth with respect to an adapted path metric and has the same form as the manifold setting.
Our approach also gives a slightly more restrictive form of Folz's theorem on conservativeness as a consequence.
\end{abstract}

\date{\today}

\maketitle
\section{Introduction}
\subsection{Brownian motion case}
The celebrated Khintchine's law of the iterated logarithm states that,
for the Brownian motion $\lp B_t\rp_{t\ge 0}$ on $\RR$,
\[\limsup_{t\rightarrow \infty} \frac{|B_t|}{\sqrt{2t \log\log t}}=1, ~~~~a.s.\]
In particular, we see that for the function $R(t)=\sqrt{(2+\e)t \log\log t}$ ($\e>0$),
\[\PP_0 \lp ~~|B_t|\leq R(t)~~\text{for all sufficiently large $t$~~}\rp =1.\]
Such a function $R(t)$ is called an upper escape rate function of the Brownian motion.

Some care has to be taken when generalizing this notion to a more general Markov process.  
Let $(V, d)$ be a locally compact
separable metric space and $\mu$ a positive Radon measure on $V$ with full support.
Let $V_{\infty}=V\cup\{\infty\}$ be the one point compactification of $(V, d)$.
If $(V,d)$ is already compact, then $\infty$ is adjoined as an isolated point.
Let $\mm=\lp \Omega,  \xt, \{\PP_x\}_{x\in V\cup\{\infty\}}, \{\ff_t\}_{t\ge 0}, \infty, \zeta\rp$ be
a $\mu$-symmetric Hunt process on $V$.
Here the sample space $\Omega$ is taken to be the space of right-continuous functions $\omega:[0,\infty]\rightarrow V_{\infty}$ such that:
\begin{enumerate}
  \item $\omega$ has
a left limit $\omega(t-)\in V_{\infty}$ for any $t\in (0,\infty)$;
  \item $\omega(t)=\infty$ for any $t\ge \zeta:=\inf\{s\ge 0: \omega(s)=\infty\}$ and $\omega(\infty)=\infty$.
\end{enumerate}
The random variable $X_t$ is defined as $X_t(\omega)=\omega(t)$ for $\omega\in \Omega$.
The random variable $\zeta$ is called the lifetime of the process $\mm$ and can be finite.
We will feel free to use some further properties of Hunt processes as presented in \cite{FOT} (Section A.2).
\begin{defi}
  \label{defi-upper-rate}
  Fix a reference point $\bar{x}\in V$.
A nonnegative
increasing
  function $R(t)$ is
  called an upper rate function for the process $\mm$, if there exists a random time $T<\zeta$ such that
  \[\PP_{\bar{x}}\lp d(X_t, \bar{x})\leq R(t)~~\text{for all
 ~~}T\le t<\zeta\rp =1.\]
\end{defi}
\begin{rem}
  \rm{We do not extend the distance $d$ to $V_{\infty}$, so $d(X_t, \bar{x})$ has no definition if $t\ge \zeta$.}
\end{rem}\bigskip

 The existence of upper rate functions can be viewed as a quantitative version of conservativeness, which means that 
 $\PP_x \lp \zeta< \infty\rp=0$
 for all $x\in V$.

Many authors studied the escape rate for the Brownian motion on a complete Riemannian manifold.
Grigor'yan \cite{Gri99-escape} initiated the study of upper rate functions in terms of volume growth of the manifold.
Grigor'yan and Hsu \cite{Gri-Hsu} obtained the sharp form of the upper rate function for Cartan-Hadamard manifolds.
Recently, Hsu and Qin \cite{Hsu-Qin} obtained the sharp result in full generality.
\begin{thm}[Hsu and Qin \cite{Hsu-Qin}]
  \label{thm-Hsu-Qin}
  Let $M$ be a complete Riemannian manifold and fix $\bar{x} \in M$. Let $B(r)$
be the geodesic ball on $M$ of radius $r$ and centered at $\bar{x}$.
Assume that
\begin{equation}\label{eq-vol-sc}
  \int^{\infty}\frac{rdr}{\log \func{vol}\lp B(r)\rp}=\infty,
\end{equation}
where $\int^{\infty}$ means that we only care about the singularity at $\infty$.
Define
\[\psi(R)=\int_{6}^{R}\frac{rdr}{\log \func{vol}\lp B(r)\rp +\log\log r}.\]
Then there is a constant $C>0$ such that $C\psi^{-1}(C t)$
is an upper rate function of Brownian motion $\xt$ on $M$.
\end{thm}
\begin{rem}\rm{
  As observed in \cite{Hsu-Qin}, (\ref{eq-vol-sc}) is equivalent to that $\lim_{R\rightarrow\infty}\psi(R)= \infty$.}
\end{rem}\bigskip
The sharpness of Theorem \ref{thm-Hsu-Qin} can be seen through examples of model manifolds.
The assumption (\ref{eq-vol-sc}) on volume growth guarantees
conservativeness (stochastic completeness/non-explosion) of the Brownian motion,
by a classical theorem of Grigor'yan \cite{GRI86-87} (see also \cite{GRI-SURVEY}, Theorem 9.1).
Roughly speaking, the borderline of volume growth satisfying (\ref{eq-vol-sc})
has the form \[\func{vol}\lp B(r)\rp=\exp \lp C r^2 \log r\log\log r\cdots \rp\]
for $r$ large enough. It is interesting to see that the same type integral
$\int\frac{rdr}{\log \func{vol}\lp B(r)\rp}$ controls conservativeness and upper rate function simultaneously.


\subsection{Main result}
The aim of this article is to prove an analogue to Theorem \ref{thm-Hsu-Qin} in the setting of continuous time 
symmetric Markov chains on locally finite weighted graphs.

We start with a gentle introduction to weighted graphs. More details will be given in the next section.


Let $(V, E)$ be a locally finite, countably infinite, connected, undirected graph without loops or multi-edges.
For abbreviation, such a graph is called a \emph{simple} graph. Here $V$ is the vertex set,
and $E$ is the edge set that is a symmetric subset of $V\times V$. For a pair $(x,y)\in E$, we write $x\sim y$ and call them neighbors.
For $x\in V$, the degree of $x$ is defined to be $\deg(x)=\#\{y\in
V: y\sim x\}$, i.e. the number of neighbors of $x$ which is finite by assumption.

By the connectivity of the graph, for any pair of distinct vertices $x, y\in V$, there always exists a sequence of vertices $x_0, \cdots, x_n$ in $V$ such that
\[x_0=x, x_n=y, x_k\sim x_{k+1} \text{~~for all~~} 0\leq k\leq n-1.\]
Such a sequence is called a path of length $n$ connecting $x$ and $y$. The graph metric $\rho$ on the graph $(V, E)$, can then be defined through
\[\rho(x,y)=\inf\{n: \text{~~~~there exists a path of length~~}n\text{~~connecting~~}x,y\}\]
for any pair of $x\neq y$. This induces the discrete topology on $V$.

A positive function $\mu : V\rightarrow(0,\infty)$ can be viewed as a Radon measure on $V$.

Let $w: V\times V\rightarrow[0,\infty)$ be a weight function on edges such that:
\begin{enumerate}
 \item $w$ is symmetric, that is, $w(x,y)=w(y,x)$  for all  $x, y \in V$;
 \item $(x,y)\in E \Leftrightarrow w(x,y)>0$.
\end{enumerate}
The triple $(V, w, \mu)$ is usually called a weighted graph.
Note that given a weighted graph $\g$, we can determine the edge set $E$ through $w$.
Throughout the paper, without further specification, we only consider weighted graphs with the underlying graph being \emph{simple}.

Define a matrix $\lp q_{xy}\rp_{x,y\in V}$ by
\[q_{xy}=\frac{w(x,y)}{\mu(x)}\]
for $x\neq y$ and
\[q_{xx}=-\lap.\] The matrix $\lp q_{xy}\rp_{x,y\in V}$ is called a $Q$-matrix and there is a unique continuous time, minimal,
 symmetric, c\`{a}dl\`{a}g Markov chain $\mm=\lp \xt, (\PP_x)_{x\in V}\rp$ associated with it (cf. \cite{Norris}, \cite{freedman}).
When there is no risk of confusion, we simply call $\xt$ the Markov chain of $\g$.
We will also talk about the conservativeness of $\g$ when we really mean that of $\xt$.

The construction of $\xt$ through a $Q$-matrix is equivalent (cf. \cite{Silverstein} Theorem 17.2)
to the construction as the Hunt process associated with the minimal (regular) Dirichlet form on $\g$,
which we will discuss in Section \ref{sect-settings}. 

For a metric $d$ on $V$, define the closed ball with center $x\in V$ of radius $r>0$ by \[B_{d}(x,r)=\{y\in V: d(x,y)\le r\}.\]
In general, we can not expect an analogue to Theorem \ref{thm-Hsu-Qin} for the Markov chain $\xt$
if we consider the volume growth of balls, $\mu \lp B_{\rho}(x,r)\rp$, in the graph metric $\rho$.
This is because
$\rho$ is determined only by the graph structure $\ve$ and is blind to the weights $\mu$ and $w$.
A useful tool is the notion of adapted metrics (or intrinsic metrics) introduced by Frank, Lenz, and Wingert \cite{FLW} and Masamune and Uemura \cite{MU}
(independently and in slightly different ways). See also Folz \cite{Folz} for the graph setting.

In this article, we will adopt a slightly stronger definition, the so-called adapted path metric \cite{HKMW}.
\begin{defi}
 Consider a weighted graph $\g$. A weight function $\sigma: E\rightarrow (0, 1]$ is called adapted if
 \begin{enumerate}
   \item $\sigma$ is symmetric, i.e.,
$\sigma(x,y)=\sigma(y,x)$, for any pair $x\sim y$;
   \item $\lap \sigma(x,y)^2\le 1$ for any $x\in V$.
 \end{enumerate}
An adapted path metric $d_{\sigma}$ is defined as

\begin{equation}\label{eq-adapted-path}
  d_{\sigma}(x,y)=\inf\{\sum_{i=0}^{n-1}\sigma(x_i, x_{i+1}): x=x_0\sim\cdots\sim x_n=y \text{~~is a path connecting~~}x,y\}.
\end{equation}
\end{defi}
\begin{rem}
  \rm{There always exists an adapted path metric on a weighted graph, though in general highly non-unique.
   For general weighted graphs, $d_{\sigma}$ as defined in (\ref{eq-adapted-path}) is only a pseudo metric.
It is always a metric on simple weighted graphs and induces the discrete topology (cf. \cite{HKMW}), due to locally finiteness and connectedness.
The restriction $\sigma\le 1$ is not serious; an upper bound suffices.}
\end{rem}\bigskip
For an adapted weight $\sigma$, the metric $d_{\sigma}$ is adapted in the sense that
\[\lap (u(x)-u(y))^2\le C^2,\forall x\in V\]
for any function $u$ on $V$ that is $C$-Lipschitz with respect to $d_{\sigma}$ for some $C>0$.
This is analogous to the property that $\vert\nabla u\vert^2\le C^2$
for $C$-Lipschitz functions with respect to the geodesic metric $d$ on a Riemannian manifold.

In general, balls in $(V, d_{\sigma})$ may not be compact (i.e. finite) (for examples cf. \cite{HKMW}).
However, if this is the case, the existence of an upper rate function implies the conservativeness of $\xt$.
Since this is kind of folklore and will not be directly used, we postpone an explanation (Lemma \ref{lem-ur-SC})
to Section \ref{sect-settings}.

The main result of this article is the following formula of upper rate function.

\begin{thm}\label{thm-main}
Let $\go$ be a simple weighted graph with an adapted path metric $d_{\sigma_o}$. Fix a reference point $\bar{x}\in V_o$.
Denote the associated Markov chain by $\lp X_t^o\rp_{t\ge0}$.
Assume that \begin{equation}\label{eq-assm-measure}
    C_{o}=\inf_{x\in V_o}\mu_o(x)>0.
  \end{equation} and that
\begin{equation}\label{eq-escape-rate-volume-growth}
\int^{\infty}\frac{r dr}{\log\mu\lp B_{d_{\sigma_o}}(\bar{x},r)\rp}=\infty.
      \end{equation}
Then we have
\begin{enumerate}
  \item $\lp X_t^o\rp_{t\ge0}$ is conservative;
  \item there exists some constant $c>0$, $\hat{R}\ge 1$,
such that
  the inverse function $\psi^{-1}(t)$ of
  \begin{equation}\label{eq-escape-rate}
    \psi(R)=c\int_{\hat{R}}^{R}\frac{rdr}{\log\mu\lp B_{d_{\sigma_o}}(\bar{x},r)\rp+\log\log r}
  \end{equation} 
is an upper rate function for $\lp X_t^o\rp_{t\ge0}$ with respect to $d_{\sigma_o}$.
\end{enumerate}
\end{thm}

\begin{rem}\rm{Combining (\ref{eq-assm-measure}) and (\ref{eq-escape-rate-volume-growth}), we see that balls in $\lp V_o, d_{\sigma_o}\rp$ have finite measure and are all finite.
Thus conservativeness follows from the existence of upper rate functions.
However, due to technical reasons that we will explain in Section 3, we have to prove conservativeness along the way of getting an upper rate function.
This explains the statement of the above theorem.

  The sharpness of the upper rate function $\psi^{-1}(t)$ has been analyzed in \cite{Huang-escape} through examples.
For a large family of symmetric weighted graphs, this is the optimal result up to the constant $c$ in (\ref{eq-escape-rate})
(see, e.g., \cite[Theorem 1.15 and Subsection 1.4]{Huang-escape}).}
\end{rem}

As a corollary of Theorem \ref{thm-main}, we obtain the following
as for Brownian motions on Riemmanian manifolds (see, e.g., \cite[Corollary 4.2]{Hsu-Qin}),
which refines the result in \cite[Example 1.17]{Huang-escape} for the exponential volume growth case.
\begin{cor}\label{cor-main}
For each volume growth condition as follows,
$\phi(t)$ is an upper rate function for $(X_t^0)_{t\geq 0}$.
\begin{enumerate}
\item $\mu(B_{d_{\sigma_o}}(\overline{x}, r))\leq Cr^D \ (D>0)$ and $\phi(t)=c\sqrt{t\log t}${\rm ;}
\item $\mu(B_{d_{\sigma_o}}(\overline{x}, r))\leq e^{Cr^{\alpha}} (0<\alpha<2)$ and $\phi(t)=ct^{1/(2-\alpha)}${\rm ;}
\item $\mu(B_{d_{\sigma_o}}(\overline{x}, r))\leq e^{Cr^2}$ and $\phi(t)=e^{ct}${\rm ;}
\item $\mu(B_{d_{\sigma_o}}(\overline{x}, r))\leq e^{Cr^2\log r}$ and $\phi(t)={\rm exp}({\rm exp}(ct))$.
\end{enumerate}
\end{cor}
\bigskip

\subsection{Idea and approach}
The conservativeness part of Theorem \ref{thm-main} is not new.
Recently, Folz \cite{FolzSC} made a breakthrough on the problem of conservativeness of weighted graphs,
by proving an analogue of the volume growth criterion of Grigor'yan.
Our result on conservativeness comes with a different proof but in a slightly weaker form.

The upper rate function part of Theorem \ref{thm-main} greatly improves the work in \cite{Huang-escape}, and provides a full analogue with the manifold case.
In \cite{Huang-escape}, a partial result was proven under the restriction that the volume growth is at most exponential.
The restriction is removed in this article by two main new ingredients, which we explain below.


The first new ingredient is a certain variant of Folz's construction for conservativeness of weighted graphs.
Roughly speaking, for each weighted graph with an adapted metric, Folz constructed a corresponding metric graph.
Then comparison can be made between the lifetime of the Markov chain on the weighted graph and the diffusion on the metric graph.
The conservativeness of the metric graph in fact implies that of the weighted graph.
The volume growth criterion of conservativeness of weighted graphs then
follows from Sturm's work (\cite{Sturm94}) on strongly local Dirichlet forms, which applies to metric graphs.

We refine Folz's construction to get more accurate comparisons. Loosely speaking, for each weighted graph $\go$ with an adapted path metric $d_{\sigma_o}$,
we associate it with a new weighted graph $\g$ by adding new vertices to $V_o$ and modifying the weights, in such a way that the original process $\lp X_t^o\rp_{t\ge0}$
is the trace of the new process $\xt$ on $V_o$.
The novelty of our construction is that it allows a uniform quantitative control of the occupation time of $\xt$ on $V_o$.

More specifically, we will consider the following type modification of a weighted graph.
\begin{defi}[Modification of a weighted graph]\label{defi-modified-graph}
   Let $\go$ be a weighted graph with an adapted path metric $d_{\sigma_o}$. We fix an orientation of $E_o$,
by choosing $\iota: E_o\rightarrow \{\pm 1\}$ which satisfies $\iota(x,y)=-\iota(y,x)$ for $(x,y)\in E_o$, and letting $E_o^+=\iota^{-1}\{1\}$.
Let $\nn: E_o\rightarrow \NN_+$ be a symmetric function such that $\nn\ge 2$. We construct a weighted graph $\g$ as follows.
  \begin{enumerate}
    \item For each $e\in E_o^+$, associate a distinct set of $\nn(e)-1$ new points \[V_e=\{x^e_1, \cdots, x^e_{\nn(e)-1}\};\]
    \item for each $e=(x,y)\in E_o^+$, we change the edge $x\sim y$ to a sequence of new edges \[x=x_0^e\sim x_1^e \sim x^e_{\nn(e)-1}\sim x^e_{\nn(e)}=y;\]
    \item define the new weight function $w$ by $w(x_i^e, x_{i+1}^e)=\nn(e) w_o(e)$ for each $e=(x,y)\in E_o^+$, $0\le i\le \nn(e)-1$ and $w=0$ otherwise;
    \item define the new weight function $\mu$ by $\mu\vert_{V_o}=\mu_o$ and $\mu(x_i^e)= \frac{2w_o(e)\sigma_o(e)^2}{\nn(e)}$ for each $e=(x,y)\in E_o^+$, $1\le i\le \nn(e)-1$.
  \end{enumerate}
 Denote the new edge set obtained in (2), (3) by $E$.
We define a new weight $\sigma$ on $E$ by setting $\sigma(x_i^e, x_{i+1}^e)=\frac{\sigma_o(e)}{\nn(e)}$ for each $e=(x,y)\in E_o^+$, $0\le i\le \nn(e)-1$.
An easy calculation shows that $\sigma$ is an adapted weight for $\g$. We call such $\g$ the modified graph of $\go$ with weight $\nn$.
\end{defi}
\begin{rem}
  \rm{Although the above construction looks a bit complicated, the geometric and probabilistic considerations behind it are clear.
Intuitively, we are splitting each jump of the process into a number of jumps with smaller steps.
The geometric relations between $\go$ and $\g$ are shown in Lemma \ref{lem-geometry}.
Briefly speaking, the adapted path metric and the volume of balls of the new weighted graph are comparable with those of the original one.

We could also have worked with metric graphs instead. Basically the metric graphs
we should look at are the generalized ones with the measure as a certain linear combination
of Dirac measures on vertices and Lebesgue measures on edges. However,
we prefer the current approach, as we stay in the category of weighted graphs to have conceptual simplicity. }
\end{rem}\bigskip


Denote by $\lp X_{t}^o\rp_{t\ge 0}$ the Markov chain of $\go$ and $\lp X_{t}\rp_{t\ge 0}$
the Markov chain of $\g$. Let $\ll_t$ be the local time of $\xt$, namely, for $x\in V$,
 \[\ll_t (x)=\int_0^t \mathbf{1}_{\{x\}}(X_s)ds,\] which can be viewed as a random measure on $V$.
 Define a PCAF $\aa_t$ (positive continuous additive functional, cf. \cite{FOT}) for $\xt$ by $\aa_t=\ll_t (V_o)$, and its right inverse $\tt_t$
by \[\tt_t=\inf\{s>0: \aa_s>t\}.\] Fix a reference point $\bar{x}\in V_o\subset V$. By the general theory of PCAF,
we show in Proposition \ref{prop-trace} that $\lp X_t^o\rp_{t\ge 0}$ has the same law as $\lp X_{\tt_t}\rp_{t\ge 0}$ under $\PP_{\bar{x}}$.
The following theorem is the key to relate $\lp X_t^o\rp_{t\ge 0}$ and $\xt$. 
\begin{thm}
  \label{thm-local-time} Let $\go$ be a weighted graph with an adapted path metric $d_{\sigma_o}$.
  Let $\g$ be a modified weighted graph with weight $\nn$ as in Definition \ref{defi-modified-graph}. Furthermore, assume that $\g$ is conservative.
Then there exists a constant $C>1$, independent of $\go$ and $\nn$ such that for some random time $\tilde{T}$,
\[  \PP_{\bar{x}}\lp \tt_{t}\le C t, \text{~~for all~~} t\ge \tilde{T}\rp=1.\]
\end{thm}

Theorem \ref{thm-local-time} implies the following:
\begin{cor}\label{cor-local-time}
Assume the same condition as in Theorem {\rm \ref{thm-local-time}}.
If $R(t)$ is an upper rate function for $\xt$, then so is  $R(C t)$ for the original Markov chain $\lp X_t^o\rp_{t\ge 0}$,
where $C>1$ is the same constant as in Theorem {\rm \ref{thm-local-time}}.
\end{cor}
Through Corollary \ref{cor-local-time},
we can work with a modified graph $\g$ to obtain an upper rate function for the original one $\go$.
It is crucial to have the constant $C$ being absolute in Theorem \ref{thm-local-time}.

To show Theorem \ref{thm-local-time} we need a large deviation type argument, which is the second new ingredient of our approach.
Proposition \ref{prop-large-deviation}, stated in a slightly weaker way, says that
\[\limsup_{t\rightarrow \infty}\frac{1}{t}\log \PP_{\bar{x}}\lp \aa_t \le \varepsilon t\rp<0,\]
which has the form of the upper bound in the large deviation principle (cf. 
\cite{Donsker-Varadhan}, 
\cite{Deuschel-Stroock}, 
\cite{Jain-Krylov}, 
\cite{Kim-Takeda-Ying}). 


Let $\ppp(V)$ be the space of probability measures over $V$, equipped with the weak topology.
Then by definition, the inequality   $\aa_t\leq \varepsilon t$ is
equivalent to saying that  the random probability measure  $\ll_t/t$ belongs to the closed subset
$\{\nu\in \ppp(V): \nu (V_{})\le \varepsilon\}$. 
However, without certain tightness condition, the upper bound applies only to compact subsets of $\ppp(V)$ in general.
Furthermore, in the generality of our setting,
we can not expect any type of tightness or the compactness of  $\{\nu\in \ppp(V): \nu (V_{})\le \varepsilon\}$.

We will develop a new way to achieve Proposition \ref{prop-large-deviation},
by directly constructing a potential $\uu$ on $V$ and a solution $\varphi$ to the corresponding Schr\"{o}dinger equation.
Roughly speaking, the sign of the potential $\uu$ naturally tells
which part of a weighted graph is visited by the Markov chain for a positive portion of time.
The minimality of the heat semigroup associated with $\g$ and the Feynman-Kac formula altogether give the desired estimate.
The construction of $\uu$ and $\varphi$ heavily depends on the structure of the modified graph.

The rest of our approach basically follows the line of argument in \cite{Gri-Hsu},
but applied to a modified weighted graph $\g$ with a carefully chosen weight $\nn$.
We use a Borel-Cantelli type argument to reduce the problem to the point-wise a priori estimate of the solution
to a heat equation on a sequence of balls.

The advantage to work with the modified graph $\g$ is that neighboring points become closer
in $d_{\sigma}$, the new adapted path metric.
Indeed, the modified weighted graph will be ``designed" in advance such that
the distance between a pair of neighboring points ``shrinks"
when the pair is far from the reference point (Lemma \ref{lem-shrinking}).
By this fact, we can construct more refined auxiliary functions in (\ref{def-aux}) than those in \cite[Subsection 5.2]{Huang-escape}.
In such a way, the technique of a priori estimate from the classical PDE works well and gives the expected result.
It should be noted that the nonlocal nature of the problem causes the lack of a chain rule and is the main obstruction in the analysis of \cite{Huang-escape}.

The assumption (\ref{eq-assm-measure}) that there is a positive lower bound on $\mu_o$ avoids the need for Sobolev inequalities in the manifold setting (cf. \cite{Gri-Hsu}) and
is important for the estimates. The subtle point is that this assumption is not preserved under the modification.
We overcome this issue again by making advantage of the structure of the modification.

The rest of this paper is organized as follows.
In Section \ref{sect-settings}, we recall the results we need from Dirichlet form theory and the related stochastic calculus.
The main technical tool, Theorem \ref{thm-local-time}, is proven in Section \ref{sect-comparison-local-time}
using the strategy of Schr\"{o}dinger equation we described before.
Several basic comparisons between weighted graphs with its modifications are also presented there.
The proof of the main result, Theorem \ref{thm-main} will be then accomplished in Section \ref{sect-main-proof}.
In  Section \ref{sect-examples}, we apply Theorem \ref{thm-main} to several weighted graphs
and obtain the upper rate functions explicitly.

We end this introduction by a few words on notations.
The letters $c$ and  $C$ (with subscript) denote finite positive constants
which may vary from place to place. If we indicate the dependence of the constant
explicitly, we write $c=c(u)$ for instance. For nonnegative functions $f(x)$ and
$g(x)$ on a space $S$, we write $f(x)\asymp g(x)$ if there exist positive constants $c_1$ and
$c_2$ such that
$$c_1 g(x)\leq f(x)\leq c_2g(x) \quad \text{for any} \ x\in S.$$

\section{Settings}\label{sect-settings}

Most of the materials in this subsection are standard. We recall the results that we need from the theory of Dirichlet forms and stochastic calculus by additive functionals.
We will follow \cite{FOT} for the general theory, and 
\cite{KL} for an analytic framework for weighted graphs with potentials.
\subsection{Analytic side}
As before, we consider a simple weighted graph $\g$ together with an adapted path metric $d_{\sigma}$.
We assume that balls in $(V, d_{\sigma})$ are all finite.
Note that all functions on $(V,d_{\sigma})$ are then in fact in the space $C(V)$, i.e. continuous. 
 We adopt the notations $\ccv$, $C_+(V)$ and $C_b(V)$ to denote the space of finitely supported functions,
the space of nonnegative functions and the space of bounded functions, respectively.
Furthermore, we consider a function $\vv\in C_+(V)$ as a potential function 
  (i.e. $\vv \cdot\mu$ as a killing measure). 
Define a positive definite, symmetric bilinear form $\lp \eee^{\vv}, \ccv\rp$ through
\[\eee^{\vv}(u,u)=\frac{1}{2}\sum_{x\in V}\sum_{y\in V} \omega(x,y)\lp u(x)-u(y)\rp^2+\sum_{x\in V} {\vv}(x)u(x)^2 \mu(x),\]
where $u\in\ccv$.
This is a closable form, and we consider its closure $ (\eee^{\vv}, \fff^{\vv})$ under $\eee_1^{\vv}=\eee^{\vv}+\parallel \cdot\parallel_{\ltwo}^2$,
in the maximal domain
\[\fff^{\vv}_{\max}=\left\{u\in L^2\lp V, \mu\rp: \sum_{x\in V}\sum_{y\in V} \omega(x,y)\lp u(x)-u(y)\rp^2 +\sum_{x\in V} {\vv}(x)u(x)^2<\infty\right\}.\]

In \cite{KL}, the regularity of the Dirichlet form $(\eee^{\vv}, \fff^{\vv})$ is shown, and the generator $\lll^{\vv}$
(we choose it to be positive definite, opposite to many authors) is determined
to be a restriction to $\ddd(\lll^{\vv})$ of the so-called formal Laplacian $\Delta^{\vv}$, which takes the form:
\begin{equation}
\label{eq-formal-lap}
\Delta^{\vv} u(x)=\lap(u(x)-u(y))+{\vv}(x)u(x), \forall x\in V.
\end{equation}

Corresponding to the self-adjoint operator $\lll^{\vv}$, there are a semigroup $\{P_t^{\vv}\}_{t\ge 0}$
and a resolvent $\{R_{\alpha}^{\vv}\}_{\alpha>0}$ on $\ltwo$, which can be understood as
\[P_t^{\vv}= \exp\lp -t \lll^{\vv}\rp, ~~~~R_{\alpha}^{\vv}=\lp \alpha +\lll^{\vv}\rp^{-1},\]
through functional calculus. They have the Markov property, and can be extended from $\ltwo$ to $C_+(V)$. Indeed, for any $u\in C_+(V)$,
choosing a sequence $\{u_n\}_{n\in \NN}\subset \ltwo$ with $0\le u_n\le u_{n+1}$ and $u_n\rightarrow u$ in the point-wise sense,
we can define \[P_t^{\vv} u(x)= \lim_{n\rightarrow \infty} P_t^{\vv} u_n (x), ~~~~R_{\alpha}^{\vv} u(x)= \lim_{n\rightarrow \infty} R_{\alpha}^{\vv} u_n (x), ~~\forall x\in V.\]
Both limits do not depend on the choice of the sequence $\{u_n\}_{n\in \NN}\subset \ltwo$, though we allow $\infty$ value in the limit.
The following result is extracted from \cite{KL} (Theorem 11 (b)) which states the minimality of the resolvent.

\begin{thm}[Keller and Lenz]
\label{thm-KL-minimality} Let $f\in C_+(V)$ and $\alpha>0$. The following two statements are equivalent:
  \begin{enumerate}
  \item There exists $g\in C_+(V)$ such that $\lp \Delta^{\vv} +\alpha\rp g\ge f$;
  \item $R_{\alpha}^{\vv} f (x)$ is finite for any $x\in V$.
  \end{enumerate}
  In this case, $R_{\alpha}^{\vv} f$ is the smallest nonnegative function $g$ with $\lp \Delta^{\vv} +\alpha\rp g\ge f$,
and it satisfies \[\lp\Delta^{\vv}+\alpha\rp R_{\alpha}^{\vv} f= f.\] 
\end{thm}

 When considering the original Dirichlet form $\lp \eee, \fff\rp$ of $\g$ (i.e. ${\vv}\equiv 0$),
we simply denote the corresponding quantities by omitting the superscript $\vv$.
The following lemma relates the Dirichlet forms $( \eee, \fff)$ and $( \eee^{\vv}, \fff^{\vv})$.
 \begin{lem}\label{lem-killed-form-domain}
   Let $\g$ be a simple weighted graph and ${\vv}\in C_+(V)$. Then
   \[\fff^{\vv}=\fff\cap L^2\lp V, {\vv} \cdot\mu\rp.\]
 \end{lem}
\begin{proof}
Define  $\tilde{\fff}^{\vv}:=\fff\cap L^2\lp V, {\vv} \cdot\mu\rp (\subseteq \fff^{\vv}_{\max})$. Then by
Theorem 6.1.2 in \cite{FOT}, $( \eee^{\vv}, \tilde{\fff}^{\vv})$ is a regular Dirichlet form with $\ccv$ as a special standard core,
since $\ccv$ is a special standard core for $\lp \eee, \fff\rp$.
By definition, $\fff^{\vv}$ is the closure of $\ccv$ in $\fff^{\vv}_{\max}$, with respect to the $\eee_1^{\vv}$-norm and hence $\fff^{\vv}=\tilde{\fff}^{\vv}$.
\end{proof}

\subsection{Probabilistic side}
Fix a simple weighted graph $\g$ with the associated regular Dirichlet form $\lp \eee, \fff\rp$.
The general theory of Dirichlet forms (\cite{FOT}) guarantees the existence and uniqueness (up to equivalence of processes)
of a corresponding Hunt process \[\mm=\lp \Omega,  \xt, \{\PP_x\}_{x\in V\cup\{\infty\}}, \{\ff_t\}_{t\ge 0}, \infty, \zeta\rp. \]
Here the filtration $\{\ff_t\}_{t\ge 0}$ is assumed to be the minimum completed admissible one.
The symbol $\infty$ here stands for the cemetery point
and $\zeta=\inf\{t>0, X_t=\infty\}$ is the lifetime.
Every function $u$ on $V$ is automatically extended to $V_{\infty}$ by setting $u(\infty)=0$.
The process constructed in this way is the same as the direct construction using Markov chain theory, as mentioned before.
The relation between the probabilistic theory and the Dirichlet form theory can be seen from the semigroup:
\[P_t u(x)= \EE_x [u(X_t)], \quad u\in C_b(V)\cap C_+(V),\]
where $u(\infty)$ is defined to be $0$.

As already mentioned, conservativeness can be formulated in terms of lifetime. The irreducibility of $\xt$ follows from the connectedness of $\ve$.
\begin{defi} Let $\g$ be a simple weighted graph with Markov chain $\mm$.
  The Markov chain $\mm$ is said to be explosive if for some/all $x\in V$,
  \[\PP_x(\zeta<\infty)>0.\] 
  Otherwise, $\xt$ {\rm (}or $\g${\rm )} is called conservative.
\end{defi}

Now we make some remarks on the relation between conservativeness and existence of an upper rate function. From the construction of $\xt$ through the $Q$-matrix (cf. \cite{Norris}, Section 2.6),
the lifetime $\zeta$ is the first time that $\xt$ performs infinitely many jumps (cf. \cite{Norris}).
Consider a finite set $K\subset V$ and define $\tau_K=\inf\{t>0 : X_t\in V\backslash K\}$ as the exit time of $K$. We can see that $\PP_x(\tau_K <\zeta)=1$ for any $x\in V$,
since $K$ is finite  and $\mm$ has no killing inside.
And we have the following: 

\begin{lem}\label{lem-ur-SC}
  Let $\g$ be a simple weighted graph with a metric $d$ which induces the discrete topology.
Fix a reference point $\bar{x}$. If all balls in $(V, d)$ are finite,
then the existence of an upper rate function $R(t)$ in $d$ implies the conservativeness
for the corresponding minimal Markov chain $\xt$.
\end{lem}

\begin{proof}
Assume the contrary, that is, $\PP_{\bar{x}}(\zeta<\infty)>0$.
Then there exists some $t_0>0$ such that $\PP_{\bar{x}}(\zeta\le t_0)>0$.
By Definition \ref{defi-upper-rate}, conditioning on the event $\{\zeta\le t_0\}$,
we have $X_{\zeta-}\in B_{d}(\bar{x}, R(t_0))$ almost surely,
since $B_{d}(\bar{x}, R(t_0))$ is finite.
However, by Lemma 4.5.2 in \cite{FOT}, we have that $X_{\zeta-}=\infty$
with probability $1$ conditioning on the event $\{\zeta<\infty\}$,
since there is no presence of killing measure. This leads to a contradiction.
\end{proof}

Three types of operations on a Hunt process are important to us: killing by a PCAF,
killing
at the
exit time,
and time change by a PCAF (cf. \cite{FOT}, chapters 4, 5, 6). In this section, we recall some basic facts about the former two operations.
The third one will be discussed in Section \ref{sect-comparison-local-time}.

Let ${\vv}\in C_+(V)$. It is clear that ${\vv} \cdot\mu$ is a positive Radon measure charging no sets of zero capacity.
We can also show that ${\vv} \cdot\mu$ is the Revuz measure corresponding
to the PCAF $\int_{0}^t {\vv}(X_s)ds$ by checking  $(5.1.13)$ in \cite{FOT}.
Moreover, combining Theorem 6.1.1 of \cite{FOT} and Lemma \ref{lem-killed-form-domain},
we have the following Feynman-Kac type formula for the semigroup $\lp P_t^{\vv}\rp_{t\ge0}$:
\begin{equation}
  \label{eq-Feynman-Kac}
  P_t^{\vv} f(x)= \EE_x \left[f(X_t)\exp\lp -\int_{0}^t {\vv}(X_s)ds\rp\right],
\end{equation}
for any $x\in V, f\in C_+(V)$.

Now more generally consider $\uu\in C_b(V)$ (not necessarily non-negative), and let $C>0$ be a constant such that $\uu+C\ge0$.
Consider the semigroup $P_t^{\uu+C}$ as defined above for the non-negative potential $\uu +C$. For $f\in C_b(V)\cap C_+(V)$, we define
 \[\tilde{P}_t^{\uu}
f(x):=\EE_x \left[f(X_t)\exp\lp -\int_{0}^t {\uu}(X_s)ds\rp \right].\]
By (\ref{eq-Feynman-Kac}), it is then direct to see that
\begin{equation}\label{eq-semigroups}
\tilde{P}_t^{\uu}
f=\exp(Ct)P_t^{\uu+C}f
\end{equation}
for any $f\in C_b(V)\cap C_+(V)$.
A combination of Theorem \ref{thm-KL-minimality} and the Feynman-Kac formula leads to the following result,
which will be crucial in the proof of Theorem \ref{thm-local-time}.

\begin{prop}
  \label{prop-minimality} Let $\uu\in C_b(V)$. If $\varphi\in C_b(V)\cap C_+(V)$ satisfies
  \[\lp\Delta+\uu\rp\varphi\ge 0,\]
then $\tilde{P}_t^{\uu} \varphi\leq  \varphi$.
\end{prop}

\begin{proof}
  Let $C>0$ be such that $\uu+C\ge0$. Then $\varphi$ satisfies that
  \[\lp\Delta+\uu+C+\alpha\rp\varphi\ge (C+\alpha)\varphi,\]
  for any $\alpha>0$.
Hence by applying
Theorem \ref{thm-KL-minimality} to the potential $\uu+C$, we have
\[(C+\alpha)R_{\alpha}^{\uu+C}
\varphi\le \varphi.\]

Let $\{\varphi_n\}_{n\in\NN}\subseteq C_+(V)\cap \ltwo$ be an increasing sequence such that
$\lim_{n\rightarrow\infty} \varphi_n=\varphi$ in point-wise sense.
By the definition of $R_{\alpha}^{\uu+C}$ and $P_t^{\uu+C}$ on $C_+(V)$, we have
   \[R_{\alpha}^{\uu+C}\varphi_n\le R_{\alpha}^{\uu+C}\varphi \le \frac{1}{C+\alpha}\varphi; \quad \lim_{n\rightarrow\infty}P_t^{\uu+C} \varphi_n=P_t^{\uu+C}\varphi.\]
   By a classical formula (cf. \cite{FOT}, (1.3.5)),
   \begin{align*}
     P_t^{\uu+C} \varphi_m &=\lim_{\alpha\rightarrow\infty}\exp(-\alpha t)\sum_{n=0}^{\infty}\frac{(t\alpha)^n}{n!}(\alpha R_{\alpha}^{\uu+C})^n \varphi_m \\
&\le\lim_{\alpha\rightarrow\infty}\exp(-\alpha t)\sum_{n=0}^{\infty}\frac{(t\alpha)^n}{n!}\alpha^n \left(R_{\alpha}^{\uu+C}\right)^n \varphi\\
&\le\lim_{\alpha\rightarrow\infty}\exp(-\alpha t)\sum_{n=0}^{\infty}\frac{(t\alpha)^n}{n!}\left(\frac{\alpha}{C+\alpha}\right)^n \varphi
     \\&=\exp(-Ct)\varphi
   \end{align*}
  for any $m\in \NN$.
By letting $m\rightarrow\infty$, we obtain
  \[P_t^{\uu+C} \varphi\le \exp(-Ct)\varphi\]
and the
assertion follows from (\ref{eq-semigroups}).
\end{proof}
Let $K\subseteq V$ be a finite subset. We define the graph theoretical closure $\cl(K)$ of $K$ by
\[\cl(K)=\{x\in V: x\in K, \text{~~or~~}\exists y\in K, \text{~~s.t.~~} x\sim y \}.\]
We call $\partial K=\cl(K)\backslash K$ the (outer) boundary of $K$.
The interior $\inte(K)$ of $K$ is defined to be the largest subset $L\subseteq K$ with $\cl(L)\subseteq K$.
This is well defined since $\cl(L_1)\cup \cl(L_2)=\cl(L_1\cup L_2)$ for all finite $L_1, L_2\subseteq V$.
Note that $\cl(\inte(K))$ may not be equal to $K$.

We can define the part (or restriction) $(\eee^K, \fff^K)$ of $(\eee, \fff)$ on $K$ by
\[\fff^K=\{u\in \fff: u\vert_{K^c}\equiv 0\}, ~~~~\eee^K(u,u)=\eee(u,u), \forall u\in \fff^K.\]
It is direct to see that $\fff^K=C(K)$ since $K$ is finite. The corresponding semigroup $P_t^K$ satisfies that
\[P_t^K u(x)=\EE_x[u(X_t)\textbf{1}_{t<\tau_K}],~~~~\forall u\in C(K), x\in K.\]

The generator of $P_t^K$ is given by
\[\Delta^K u(x)=\lap u(x)-\frac{1}{\mu(x)}\sum_{y\in K}w(x,y) u(y),~~~~\forall u\in C(K), x\in K.\]
Given a function $u\in C(K)$, we see that
\[\Delta^K u(x)=\Delta u(x),~~~~\forall x\in \inte(K).\]
The following result is standard (for our purpose, it was treated in \cite{Huang-escape}, Proposition 2.8).
\begin{prop}
  \label{prop-heat-solution}Let $\g$ be a simple weighted graph.
   Let $K\subseteq V$ be finite. Define a function $u$ on $K\times[0,\infty)$ 
to be
\[u(x,t)=\PP_x \lp \tau_{K}\leq t\rp.\]
Then $u(x,t)$ is differentiable in $t$ on $[0,\infty)$ and satisfies
\[        \ddt u(x,t)+\Delta u(x,t)=0, \]
for all $x\in \inte(K)$ and $t\geq 0$ with initial condition $u(\cdot,0)\equiv0$ on $K$.
\end{prop}
\begin{rem}
 \rm{Note that $u(x,t)=\PP_x \lp \tau_{K}\leq t\rp =1-P_t^K \textbf{1}(x)$.}
\end{rem}\bigskip

\subsection{Integral maximum principle}
The technical tool to estimate the solution $u$ in Proposition \ref{prop-heat-solution} is the so-called integral maximum principle.
This kind of technique is classical in the parabolic PDE theory and dates at least back to Aronson \cite{ARO67}.
See also Grigor'yan \cite{Gri-integrated} for the manifold setting. In the present context, it is developed in \cite{Huang-thesis}.
We state it here with a slight modification to be compatible with our notations.
\begin{prop}[Integral maximum principle]\label{prop-integral-maximum-principle}
Let $\g$ be a simple weighted graph.
Let $L\subseteq V$ be a finite subset of $V$. Let $K\subseteq \inte(L)$ be nonempty.

 Fix some $T>0$.  Let
$u(x,t)$ be a function on $L\times [0, T]$ that is
differentiable in $t$ on $[0,T]$ and $u(x,0)\equiv 0$. Assume further that $u(x,t)$ solves the heat equation
\begin{equation}\label{equation-localized-heat-equation}
\frac{\partial}{\partial t}u(x,t)+\Delta u(x,t)=0,
\end{equation}
on $K\times [0, T]$.

Take two auxiliary functions $\eta(x)$ on
$L$ and $\xi (x, t)$ on $L\times [0,T]$ such that
\begin{enumerate}
\item the function $\eta(x)\geq 0$ is finitely supported and ${\rm{supp}} \eta\subseteq K$;

\item $\xi (x, t)$ is continuously differentiable in $t$ on $[0,T]$
for each $x\in L$;

\item the inequality
$$(\eta^2 (x)-\eta^2 (y))(e^{\xi(x,t)}-e^{\xi(y,t)})\geq 0$$
holds for all $x\sim y, x,y\in L$ and $t\in [0,T]$;

\item the inequality $$\mu(x)\frac{\partial}{\partial t}\xi(x,t)+\frac{1}{2}\sum_{y\in L}w(x,y)(1-e^{\xi(y,t)-\xi(x,t)})^2 \leq 0$$
holds for any $x\in L$ and $t\in [0,T]$.
\end{enumerate}
Then for any $s\in(0,T]$, we have the following estimate:
\begin{equation}
\sum_{x\in
K}u^2(x,s)\eta^2(x)e^{\xi(x,s)}\mu(x)\leq2\int_{0}^{s}\sum_{x\in
L}\sum_{y\in
L}w(x,y)(\eta(x)-\eta(y))^2u^2(y,t)e^{\xi(x,t)}dt.
\end{equation}

\end{prop}

\section{Comparison of upper rate functions}\label{sect-comparison-local-time}
Fix a simple weighted graph $\go$ with an adapted path metric $d_{\sigma_o}$.
Let $\g$ be the modified weighted graph for weight $\nn$ as in Definition \ref{defi-modified-graph}.
It is direct to see that the underlying graph $(V, E)$ of $\g$ is simple.
Recall that $d_{\sigma}$ is the adapted path metric on $\g$ as constructed in Definition \ref{defi-modified-graph}.

\subsection{Basic properties of the modified weighted graph}
 We state some basic properties of $\g$ and its relation with the original graph $\go$.

\begin{lem}
  \label{lem-geometry}The following relations hold.
  \begin{enumerate}
    \item The metric $d_{\sigma}$ satisfies that $d_{\sigma}\vert_{V_o\times V_o}=d_{\sigma_o}$.
    \item Fix $x_0\in V_o$. The measure of a ball $B_{d_{\sigma_o}}(x_0, r)$ as a subset of $V_o$ and the measure of $B_{d_{\sigma}}(x_0, r)$ satisfies:
        \[\mu_o(B_{d_{\sigma_o}}(x_0, r))\le \mu(B_{d_{\sigma}}(x_0, r))\le 3\mu_o(B_{d_{\sigma_o}}(x_0, r)).\]
  \end{enumerate}
\end{lem}
\begin{proof}
(1) \ Note that for $x,y\in V_o\subseteq V$, a path of edges in $E$ connecting $x$ and $y$ necessarily has the following form
\[x=x_0\sim \cdots \sim x_i^{e_0}\sim \cdots \sim x_1\sim \cdots \sim x_{n-1}\sim \cdots\sim x_{k}^{e_{n-1}} \sim \cdots\sim x_{n}= y,\]
where $x_0, \cdots x_{n}\in V_o$ and $e_0,\cdots e_{n-1}\in E_o$. The length of such a path with respect to $\sigma$ is the same of the length in $\sigma_o$ of the path
$x_0\sim \cdots \sim x_k\sim \cdots\sim x_{n}$, understood in $(V_o, E_o)$.
By the definition of the path metric,
(1) holds.

\noindent
(2) \ It follows that $B_{d_{\sigma_o}}(x_0, r)=B_{d_{\sigma}}(x_0, r)\cap V_o$.
Since $\mu\vert_{V_o}=\mu_o$, we have
\[\mu_o(B_{d_{\sigma_o}}(x_0, r))\le \mu(B_{d_{\sigma}}(x_0, r)).\]

Consider $x\in V_o$ with $\deg(x)=n$. Let $e_1, \cdots, e_n\in E_o$ be the edges with $x$ as a vertex and set $V_x=\bigcup_{k=1}^{n}V_{e_k}$.
Since
$$\mu(x_i^e)= \frac{2w_o(e)\sigma_o(e)^2}{\nn(e)}$$ for $e\in E_o$, $1\le i\le \nn(e)-1$,
we have
\[\mu(V_x)=\sum_{k=1}^{n}\sum_{i=1}^{\nn(e_k)-1}\mu(x_i^{e_k})\le \sum_{k=1}^{n} 2w_o(e_k)\sigma_o(e_k)^2\le 2\mu_o(x)\]
by the adapted-ness of $\sigma_o$. The last inequality follows by observing that
\[B_{d_{\sigma}}(x_0, r)\subseteq B_{d_{\sigma_o}}(x_0, r)\cup \bigcup_{x\in B_{d_{\sigma_o}}(x_0, r)}V_x.\]
\end{proof}
\begin{rem}\rm
(i) \  If any ball in $(V_o, d_{\sigma_o})$ is finite,
then so is that in $(V, d_{\sigma})$ by (1).

\noindent
(ii) \  Suppose that $\go$ satisfies the volume growth condition (\ref{eq-escape-rate-volume-growth}) with $d_{\sigma_o}$.
Then by (2), so does $\g$ with $d_{\sigma}$.  Furthermore, by setting
$$\psi_o(R)=\int_{\hat{R}}^{R}\frac{r dr}{\log\mu_o\lp B_{d_{\sigma_o}}(\bar{x},r)\rp}
\quad \hbox{and} \quad  \psi(R)=\int_{\hat{R}}^{R}\frac{r dr}{\log\mu\lp B_{d_{\sigma}}(\bar{x},r)\rp},$$
we have $\psi\le \psi_o\le C\psi$ for some $C>1$.
\end{rem}\bigskip

\begin{prop}\label{prop-trace}
The Dirichlet form $(\eee_o, \fff_o)$ of the original weighted graph $\go$ is the trace of $(\eee, \fff)$
with respect to the Revuz measure $\mu_o=\textbf{1}_{V_o}\cdot\mu$.
\end{prop}
\begin{proof}
We denote by $\lp\breve{\eee}, \breve{\fff}\rp$ the trace of $(\eee, \fff)$ for the Revuz measure $\mu_o$.
In other words,
\begin{equation} \label{eq:1}
\left\{ \begin{aligned}
         \breve{\fff} &= \{u\in L^2(V_o,\mu_o) : u=v \vert_{ V_o} \text{~~for some~~} v\in \fff_e\} \\
               \breve{\eee}(u, u)   &= \eee(H_{V_o}v, H_{V_o}v ), u\in \breve{\fff}, v\in \fff_e, u=v\vert_{V_o}.
                          \end{aligned} \right\}
                          \end{equation}
(see \cite[Section 6.2]{FOT}). Here $(\fff_e, \eee)$
is the extended Dirichlet space of $(\eee, \fff)$ (see \cite[p.41]{FOT} for definition).
For $v\in \fff_e$, $H_{V_o}v$ is defined by
\[H_{V_o}v(x)=\EE_x[v(X_{\varsigma_{V_o}});\varsigma_{V_o}<\infty],\]
where $\varsigma_{V_o}=\inf\{t>0:X_t\in V_o\}$ is the hitting time of $V_o$.
Let $u=v\vert_{V_o}$. Then it is clear that $H_{V_o}v\vert_{V_o}=u$.

Now fix $z\in V_e\subseteq V\backslash V_o$ for some $e=(x,y)\in E^+_o$.
Since
\[\PP_z(\varsigma_{V_o}<\infty,X_{\varsigma_{V_o}}\in\{x,y\} )=1,
\]
we
have by the strong Markov property,
\begin{align*}
 H_{V_o}v(z)&= \EE_z[v(X_{\varsigma_{V_o}});\varsigma_{V_o}<\infty] \\
  &=u(x)\PP_z(X_{\varsigma_{V_o}}=x)+u(y)\PP_z(X_{\varsigma_{V_o}}=y).
\end{align*}
Again by the strong Markov property, the function $f_1(k):=\PP_{x_k^e}(X_{\varsigma_{V_o}}=x)$
satisfies
\[f_1(k)=\frac{f_1(k-1)+f_1(k+1)}{2}, \] 
for $1\le k\le \nn(e)-1$ with $f_1(0)=1$, $f_1(\nn (e))=0$.
A similar relation also holds for $f_2(k):=\PP_{x_k^e}(X_{\varsigma_{V_o}}=y)$ with $f_2(0)=0$, $f_2(\nn (e))=1$.
Since these relations imply that
$$f_1(k)=1-\frac{k}{\nn(e)}, \quad f_2(k)=\frac{k}{{\nn(e)}}$$
we obtain for $z=x_k^e$,
\begin{equation}\label{eq-harmonic}
H_{V_o}v(z)=\left(1-\frac{k}{\nn(e)}\right)u(x)+\frac{k}{\nn(e)}u(y).
\end{equation}

By the definition of the weighted graph (Definition \ref{defi-modified-graph}),
for any $v\in \fff_e$ with $u=v\vert_{V_o}\in  L^2(V_o,\mu_o)$,
\begin{align*}
\breve{\eee}(u,u)&=  \eee(H_{V_o}v, H_{V_o}v )
  \\&=\frac{1}{2}\sum_{x\in V}\sum_{y\in V} w(x,y)\lp H_{V_o}v(x)-H_{V_o}v(y)\rp^2\\
  &=\sum_{e=(x,y)\in E_o^+}\sum_{k=0}^{\nn(e)-1}w_o(x,y)\nn(e)\lp H_{V_o}v(x_k^e)-H_{V_o}v(x_{k+1}^e)\rp^2.
\end{align*}
Then by (\ref{eq-harmonic}), the right hand side above is equal to
\begin{align*}
\sum_{e=(x,y)\in E_o^+}\sum_{k=0}^{\nn(e)-1}w_o(x,y)\nn(e)\lp \frac{u(x)-u(y)}{\nn(e)}\rp^2
&=\sum_{e=(x,y)\in E_o^+} w_o(x,y)(u(x)-u(y))^2\\
&=\eee_o(u,u).
\end{align*}
Moreover,
since $C_c(V)$ is a special standard core of $\lp \eee, \fff\rp$,
it follows by  Theorem 6.2.1 (iii) in \cite{FOT} that
$C_c(V)\vert_{V_o}$ ($=C_c(V_o)$) is a core of $\lp\breve{\eee}, \breve{\fff}\rp$.
As a result, we have $\fff_o=\breve{\fff}$ and hence $\lp\eee_o, \fff_o\rp=\lp\breve{\eee}, \breve{\fff}\rp$.
\end{proof}

By the general theory of time changed process (cf. \cite{FOT}, Section 6.2),
we have the following from Proposition \ref{prop-trace}

\begin{cor}\label{cor-trace}
The original Markov chain $\lp \lp X_t^o\rp_{t\ge 0}, \lp \PP^o_x\rp_{x\in V_o}\rp$ is
the time changed process of $\lp \xt, \lp \PP_x\rp_{x\in V}\rp$ with respect to the PCAF $\aa_t=\ll_t (V_o)$.
Namely, for a fixed reference point $\bar{x}\in V_o$,
$\lp X_t^o\rp_{t\ge 0}$ has the same law as $\lp X_{\tt_t}\rp_{t\ge 0}$ under $\PP_{\bar{x}}$,
and the lifetime $\zeta_o$ of $\lp X_t^o\rp_{t\ge 0}$ is $\aa_{\zeta}$.
\end{cor}


\subsection{Solution to a Schr\"{o}dinger equation}

We construct a potential $\uu\in C_b(V)$ and the solution $\varphi\in C_b(V)\cap C_+(V)$
to the corresponding Schr\"{o}dinger equation, with certain nice properties.

\begin{prop}\label{prop-uu-varphi}
  There exists $\uu\in C_b(V)$ with 

  \makeatletter
  \let\@@@alph\@alph
  \def\@alph#1{\ifcase#1\or \or $'$\or $''$\fi}\makeatother
  \begin{subnumcases}
  {\uu (x)=}
  -C_1, &$x\in V\backslash V_o$, \label{eq:c1}\\
  C_2, &$x\in V_o$,\label{eq:c2}
  \end{subnumcases}
  \makeatletter\let\@alph\@@@alph\makeatother
  for some constants $C_1,C_2 >0$, and $\varphi\in C(V)$
with $1\le \varphi\le C_3$ for some $C_3>1$ such that
 \[\lp\Delta+\uu\rp \varphi\ge0.\] The constants $C_1, C_2, C_3$ are absolute.
\end{prop}
\begin{proof}
  Let $\nn_0=\inf\{\nn(e):e\in E_o^+\}\ge 2$. Define $M_1>1$ to be the minimum of $M$ such that the following inequalities hold for any $\theta\in [0,1/2]$:

  \begin{equation} \label{eq:2}
  \left\{ \begin{aligned}
           \frac{\theta}{M} &\le \sin \theta\le \theta \le \tan \theta\le M\theta; \\
                    \frac{\theta^2}{2M^2}&\le 1-\cos\theta\le \frac{\theta^2}{2}.
                            \end{aligned} \right.
                            \end{equation}
 Define $M(\nn_0)$ to be the minimum of $M$ such that the above inequalities hold for any $\theta\in [0, \frac{1}{\nn_0}]$. It is clear that $M(\nn_0)\le M_1$.
In the following, we fix some $M_0\ge M(\nn_0)$.

 Let $C_1=\frac{1}{2M_0^2}$. For each $e\in E_o^+$, we set $\theta(e)\in (0,\pi/2)$ by
 \[\cos (\theta(e)) =1-\frac{C_1 \sigma_o(e)^2}{\nn(e)^2}.\]
Note that
$$\frac{M_0\sqrt{2C_1}\sigma_o(e)}{\nn(e)}=\frac{\sigma_o(e)}{\nn(e)}\le \frac{1}{\nn_0}$$
because $\sigma_o$ is an adapted weight.
Then we have by (\ref{eq:2}),
\[\cos \left(\frac{\sigma_o(e)}{\nn(e)}\right)\le 1-\frac{C_1 \sigma_o(e)^2}{\nn(e)^2}=\cos(\theta (e)),\]
which implies that
$$\theta(e)\le\frac{\sigma_o(e)}{\nn_0}\le \frac{1}{\nn_0}.$$

Let $C_2 =
C_1 M_1 M_0^2
=\frac{M_1}{2}$
and define the potential $\uu$ as in Proposition \ref{prop-uu-varphi}. Now we start constructing the super-solution $\p$ on $V$.

First, for $x\in V_o\subseteq V$ we simply set $\p(x)=1$.
For $x_k^e\in V_e$, the set of new points, for some $e=(x,y)\in E_o^+$, with $1\le k\le \nn(e)-1$,
we set
\[\p(x_k^e)=\frac{\sin \lp k\theta(e)+\frac{\pi-\nn(e)\theta(e)}{2}\rp}{\cos \lp \frac{ \nn(e)\theta(e)}{2}\rp}.\]
Note that the definition is compatible when we take $k=0$ or $k=\nn(e)$ in the above,
and that $\p(x_k^e)=\p(x_{\nn(e)-k}^e)$ holds.

By the estimate $\nn(e)\theta(e)\le 1$, we have that for all $0\le k\le \nn(e)$,
\[\frac{\pi}{4}\le \frac{\pi-\nn(e)\theta(e)}{2}\le \frac{\pi}{2}; ~~\frac{\pi}{4}\le k\theta(e)+\frac{\pi-\nn(e)\theta(e)}{2}\le\frac{3\pi}{4}.\]
It follows that
$$1\le\p\le \frac{1}{\cos(1/2)}<\sqrt{2}\ (=C_3).$$

We are left to check that $\p$ is indeed a super-solution.
We divide the verification into two cases.

Case 1: Let $x_k^e\in V_e$, for some $e=(x,y)\in E_o^+$, with $1\le k\le \nn(e)-1$. Consider the elementary identity
\[2\sin (k\theta+\Phi) \cos\theta=\sin((k-1)\theta+\Phi)+\sin((k+1)\theta+\Phi),\]
where we can take $\theta=\theta(e)$ and $\Phi=\frac{\pi-\nn(e)\theta(e)}{2}$.
This gives
\[2\p (x_k^e)-\frac{2C_1 \sigma_o(e)^2}{\nn(e)^2}\p (x_k^e) =2 \cos(\theta(e))\p (x_k^e) =\p (x_{k-1}^e)+\p (x_{k+1}^e),\]
which is simply
\[\Delta \p(x_k^e)=C_1 \p(x_k^e).\]

Case 2: Let $x\in V_o$ and $e=(x,y)\in E_o^+ $. We then have by (\ref{eq:2}),
\begin{align*}
  \p(x_1^e)-\p(x)&=\sin (\theta(e))\cdot\tan \left(\frac{\nn(e)\theta(e)}{2}\right)+\cos(\theta(e))-1\\
    &\le \frac{M_1 \nn(e)\theta(e)^2}{2}
    \le M_1 \nn(e)M_0^2(1-\cos(\theta(e)))\\&=M_1 \nn(e)M_0^2\cdot\frac{C_1 \sigma_o(e)^2}{\nn(e)^2}= \frac{C_2 \sigma_o(e)^2}{\nn(e)}.
\end{align*}
Since $\p(x_1^e)=\p(x_{\nn(e)-1}^e)$, the same estimate
\[\p(x_{\nn(e)-1}^e)-\p(x)\le \frac{C_2 \sigma_o(e)^2}{\nn(e)}\]
holds if $e=(y,x)\in E_o^+ $.
We thus arrive at the inequality
\begin{align*}
  -\Delta\p(x)&\le \frac{1}{\mu_o(x)}\sum_{y\in V_o, e=(x,y)\in E_o}w_o(x,y)\nn(e)\cdot\frac{C_2 \sigma_o(e)^2}{\nn(e)}\\
  &=C_2\cdot\frac{1}{\mu_o(x)}\sum_{y\in V_o, (x,y)\in E_o}w_o(x,y)\sigma_o(e)^2\\
  &\le C_2=\uu(x)\p(x),
\end{align*}
by the adapted-ness of $\sigma_o$.

As a consequence, the assertion holds with constants $C_1=\frac{1}{2M_1^2}, C_2=\frac{M_1}{2}, C_3=\sqrt{2} $.
\end{proof}
\subsection{Feynman-Kac formula and a large deviation type argument}
In this subsection, we assume that $\g$ is \emph{conservative}.
\begin{prop}
  \label{prop-large-deviation}
Let $C_1$ and $C_2$ be the same constants as in Proposition {\rm \ref{prop-uu-varphi}}. Then for any
$\varepsilon\in \left(0,\frac{C_1}{C_1+C_2}\right)$ and $\delta>0$, there exist positive constants $c_0, C$ such that for any $t>0$,
  \[\sup_{x\in V}\PP_{x}\lp \aa_t \le \varepsilon t+\delta \rp\le C\exp\lp -c_0 t\rp.\]
\end{prop}
\begin{proof}

Let $\uu$, $\varphi$ and $C_1, C_2, C_3$ be as in Proposition \ref{prop-uu-varphi}.
Recall that $\ll_t=\int_0^t \textbf{1}_{\cdot}(X_s)ds$ and $\aa_t=\ll_t(V_o)$.

By Proposition \ref{prop-minimality} and the Feynman-Kac formula, we have
for any $x\in V$ and $\varepsilon\in \left(0,\frac{C_1}{C_1+C_2}\right)$,
\begin{align*}
C_3\ge  \p(x)&\ge
\tilde{P}_{t}^{\uu} \p(x)\\
  &=\EE_x \left[\p(X_t)\exp\lp -\int_{0}^t {\uu}(X_s)ds\rp \right]\\
  &=\EE_x \left[\p(X_t)\exp\lp -\int_{V} {\uu}d\ll_t\rp\right].
\end{align*}
Since we have $\ll_t(V)=t$  by assumption,
the right hand side of the inequality above is equal to
\begin{align*}
&\EE_x [\p(X_t)\exp\lp -C_2 \ll_t(V_o) +C_1 (t-\ll_t(V_o)) \rp ]\\
  &\ge \EE_x [\p(X_t)\exp\lp -C_2 \ll_t(V_o) +C_1 (t-\ll_t(V_o)) \rp; \ll_t(V_o)\le\varepsilon t +\delta ]\\
  &\ge \inf_{x\in V}\p(x)\cdot \exp \{t\lp C_1-\varepsilon (C_1+C_2)\rp -\delta (C_1+C_2)\}\PP_{x}\lp \aa_t \le \varepsilon t+\delta\rp.
\end{align*}
Therefore, the assertion follows by
letting $c_0= C_1-\varepsilon (C_1+C_2)$, and $C=C_3\exp\lp \delta (C_1+C_2)\rp$.
\end{proof}
\begin{rem}
  \rm{The conservativeness is required to ensure $\ll_t(V)=t$. In the explosive case,
 although $\ll_t(V_{\infty})=t$, the above argument does not work as $\p(\infty)=0$.}
\end{rem}\bigskip
\begin{cor}\label{cor-At-estimate}

For any $\varepsilon\in\left(0,\frac{C_1}{C_1+C_2}\right)$ and
$\kappa>0$, there is a random time $\hat{T}$ such that
  \[\PP_{\bar{x}}\lp \aa_t> \varepsilon t+\kappa, \text{~~for all~~}t\geq\hat{T}
\rp=1.\]

\end{cor}
\begin{proof}
Let $\delta=\kappa/(1-\e)$
and fix $\varepsilon\in\left(0,\frac{C_1}{C_1+C_2}\right)$.
By Proposition \ref{prop-large-deviation}, there exist positive constants $c_0, C$ such that
\begin{equation}\label{ineq-exp-decay}
\PP_{\bar{x}}\lp \aa_t \le \varepsilon t +\delta \rp\le C\exp(-c_0 t).
\end{equation}

Let $t_n=n\delta$ for $n\in\NN$
and define  a sequence of events $\lp E_n\rp_{n\in \NN}$ by
\[E_n=\{\omega\in \Omega: \aa_{t_n}(\omega) \le \varepsilon t_{n}+\delta\}.\]
Then it follows by (\ref{ineq-exp-decay}) that

\[\sum_{n\in \NN}\PP_{\bar{x}}\lp E_n\rp<\infty.\]
Thus by the Borel-Cantelli lemma,
there is an $\NN$-valued random variable $N$, such that for almost all $\omega\in \Omega$, $\omega\in E_n^c$ for any $n\ge N(\omega)$.

Define a random time $\hat{T}=t_{N}$.
Then for any $t\ge t_{N}$ with some $n\ge N$ such that $t_n\le t<t_{n+1}$,
we have \[\aa_t\ge \aa_{t_n}>\varepsilon t_{n}+\delta=\varepsilon t_{n+1}+(1-\e) \delta>\varepsilon t+\kappa,\] almost surely,
whence the assertion follows.
\end{proof}
\begin{rem}\label{rem-SC}
  \rm{Note that the conservativeness of $\g$ implies that of $\go$ by this result, since $\zeta_o=\aa_{\zeta}=\aa_{\infty}=\infty$.
  Unfortunately, we do not know whether Corollary \ref{cor-At-estimate} holds without the assumption on conservativeness.}
\end{rem}\bigskip

We are now ready to prove Theorem \ref{thm-local-time}.
\begin{proof}[Proof of Theorem \ref{thm-local-time}]
We continue using the notations in the previous two results.
Fix
$\varepsilon\in\left(0,\frac{C_1}{C_1+C_2}\right)$,
and consider the random time $\hat{T}$ such that
  \[\PP_{\bar{x}}\lp \aa_t> \varepsilon t, \text{~~for all~~} t\geq \hat{T}
\rp=1.\]
By the definition of right continuous inverse, $\aa_t>\varepsilon t$ implies that
$\tt_{\varepsilon t}\le t.$
We thus have
\[\PP_{\bar{x}}\lp \aa_t> \varepsilon t, \text{~~for all~~} t\geq \hat{T}
\rp\le \PP_{\bar{x}}\lp \tt_{\e t}\le t, \text{~~for all~~} t\geq \hat{T}
\rp.\]
Rename $s=\e t$ and define $ \tilde{T}=\e\hat{T} $.
Then since
\[\PP_{\bar{x}}\lp \tt_{s}\le \frac{1}{\e} s, \text{~~for all~~} s\geq \tilde{T}
\rp=1,\]
the assertion holds for any $C> 1+\frac{C_2}{C_1}$.
\end{proof}
\section{Proof of the main theorem}\label{sect-main-proof}
In this section, we accomplish the proof of Theorem \ref{thm-main} in several steps.
Here we first summarize the overall structure of the proof.

Let $\go$ be a simple weighted graph with an adapted path metric $d_{\sigma_o}$
and fix $\bar{x}\in V_o$.
 Assume that (\ref{eq-assm-measure}) and (\ref{eq-escape-rate-volume-growth}) hold for $\go$. Namely, we assume that
 \[C_o=\inf_{x\in V_o}\mu_o(x)>0,\]
 and
 \[\int^{\infty}\frac{r dr}{\log\mu_o\lp B_{d_{\sigma_o}}(\bar{x},r)\rp}=\infty.\]
First we design the modified weighted graph $\g$ by specifying the weight function $\nn$.
As shown by Lemma \ref{lem-geometry}, the volume growth condition also holds for $(V, d_{\sigma}, \mu)$.
We apply the integral maximum principle (Proposition \ref{prop-integral-maximum-principle})
to obtain estimates for the so-called crossing time for the new Markov chain  $\xt$.
The Borel-Cantelli lemma (applied in Lemma \ref{lem-SC-modified-graph}) then leads to the formula of upper rate function for $\xt$, and its conservativeness as a by product.
By Remark \ref{rem-SC}, we obtain the conservativeness of the original process $\lp X_t^o\rp_{t\ge 0}$.
Conservativeness of $\xt$ also allows us to apply Theorem \ref{thm-local-time}.
Then the same formula of the upper rate function (up to a different constant $c>0$) holds for the original process.

\subsection{Design the modified graph}

Let $f(r)=\log\mu_o\lp B_{d_{\sigma_o}}(\bar{x},r)\rp-\log C_o\ge 0$ for $r\ge 0$.
We now specify the choice of $\nn:E_o\rightarrow \NN$.
Set $R_n=2^{n+4}$ for $n\in \NN$ and
$$\sigma_n=\frac{1}{f(R_n)+2+\log\log R_n}.$$
Choose $\nn$ so that
\begin{equation}\label{n(e)}
\nn(e)\ge f(R_n)+2+\log\log R_n
\end{equation}
for any $e=(x,y)\in E_o$ such that $d_{\sigma_o}(x, \bar{x})\vee d_{\sigma_o}(y, \bar{x})\ge 2^{n+2}-1$.

Let $\g$ be the modified weighted graph for weight $\nn$ with the adapted path metric $d_{\sigma}$ as in Definition \ref{defi-modified-graph}.
The following result is crucial for our estimates later.
\begin{lem}\label{lem-shrinking}
  Let $e=(x,y)\in E_o^+$. If
\[d_{\sigma}(x_k^e, \bar{x})\vee d_{\sigma}(x_{k+1}^e, \bar{x})\ge 2^{n+2}\]
for some $0\le k\le \nn(e)-1$, then
  \[\sigma(x_{k}^e,x_{k+1}^e)\le \sigma_n.\]
\end{lem}
\begin{proof}
  By the definition of the adapted path metric, we have
  \[d_{\sigma}(x, \bar{x})\vee d_{\sigma}(y, \bar{x})\ge d_{\sigma}(x_k^e, \bar{x})\vee d_{\sigma}(x_{k+1}^e, \bar{x})-1\ge 2^{n+2}-1.\]
Since we see by Lemma \ref{lem-geometry} that
$d_{\sigma_o}(x, \bar{x})\vee d_{\sigma_o}(y, \bar{x})=d_{\sigma}(x, \bar{x})\vee d_{\sigma}(y, \bar{x})$,
it follows by the choice of $\nn$ that $\sigma(x_{k}^e,x_{k+1}^e)=\frac{\sigma_o(x,y)}{\nn(e)}\le \sigma_n .$
\end{proof}


\subsection{Borel-Cantelli argument}

We basically follow the argument in \cite{Gri99-escape} and
\cite{Gri-Hsu} for upper rate functions of the Brownian motion on Riemannian manifolds.

Set $R_n=2^{n+4}$ for $n\in \NN$ as before.
Denote the balls $B_{d_{\sigma}}(\bar{x},R_n)$ by $\hat{B}_n$.
Let $B_n\subseteq \hat{B}_n$ be certain subsets to be chosen later, such that $(B_n)_{n\in \NN}$ is increasing.
Let $\tau_n$ be the exit time of $B_n$, that is, $\tau_n=\tau_{B_n}=\varsigma_{V\backslash B_n}$.

We start with the following standard lemma
which reduces the problem to certain estimates on the so-called crossing time $\tau_n-\tau_{n-1}$.
\begin{lem}
  \label{lem-SC-modified-graph} 
and set $\tau_n=\tau_{B_n}=\varsigma_{V\backslash B_n}$.
 Suppose for a sequence of positive numbers $\{c_n\}_{n=1}^{\infty}$ with
$\sum_{n=1}^{\infty} c_n=\infty$,
we have that
\begin{equation}
  \label{eq-Borel-Cantelli}\sum_{n=1}^{\infty}\PP_{\bar{x}}(\tau_n-\tau_{n-1}\leq c_n
)<\infty.
\end{equation}
Then $\g$ is conservative.
Suppose further that we
can find a strictly increasing homeomorphism
$\psi:\RR_+\rightarrow\RR_+$ such that
\begin{equation}\label{equation-psi}
    t_{n-1}-\psi(R_n)\rightarrow +\infty
\end{equation}
as $n\rightarrow\infty$, where $t_n=\sum_{k=1}^{n}c_k$. Then
$\psi^{-1}(t)$ is an upper rate function
for $\xt$.
\end{lem}
\begin{proof}
By the Borel-Cantelli lemma,
there exists a subspace $\Omega_1\subset\Omega$
with $\PP_{\overline{x}}(\Omega_1)=1$ and an $\NN_+$-valued random variable $N$ such that,
for any $\omega\in \Omega_1$,
\[\tau_n(\omega)-\tau_{n-1}(\omega)> c_n\]
for all $n\ge N(\omega)$.
As we already mentioned in Section \ref{sect-settings},
there exists a subspace $\Omega_2\subset\Omega$ with $\PP_{\overline{x}}(\Omega_2)=1$ such that
for any $\omega\in \Omega_2$,
we have $\zeta(\omega)>\tau_n(\omega)$  for all $n\in \NN$.
Hence for any $\omega\in \Omega_1\cap \Omega_2$,
\begin{equation}\label{ineq-life}
\zeta(\omega)>\tau_n(\omega)\ge \tau_n(\omega)-\tau_{N(\omega)-1}(\omega)\ge c_{N(\omega)}+\cdots c_n
\end{equation}
for all $n\ge N(\omega)$ . This implies that $\zeta=\infty$ almost surely because $\PP_{\overline{x}}(\Omega_1\cap \Omega_2)=1$.

 Let $T_0=c_1+\cdots+c_{N}$. Then by (\ref{equation-psi}),
there exists an $\NN_+$-valued random variable $N'\ge N+1$ such that for any
$n\ge N'$,
\begin{equation}\label{ineq-t-psi}
t_{n-1}-\psi(R_n)>T_0,
\end{equation}
$\PP_{\bar{x}}$ almost surely.

Let $T= \psi(R_{N'})$. Then for any $t\ge T$ with
\[\psi(R_{n-1})<t\leq\psi(R_n)\]
for some $n\ge N'$,
we have by (\ref{ineq-t-psi}) and (\ref{ineq-life}),
\[t\le\psi(R_n)<t_{n-1}-T_0<\tau_{n-1}.\]
Hence it follows that
\begin{equation}\label{equation-ur-psi-inverse}
d(X_t,\bar{x})\leq R_{n-1}\leq\psi^{-1}(t).
\end{equation}
Notice that $\lim_{n\rightarrow\infty}\psi(R_n)=\infty$, so
(\ref{equation-ur-psi-inverse}) holds for all $t\ge T$,
$\PP_{\bar{x}}$ almost surely. In other words, $\psi^{-1}(t)$ is an upper rate function
for $\xt$.
\end{proof}

Now let us specify the choice of $\{B_n\}_{n\in \NN}$.
Recall the notions of graph theoretical interior and boundary.
To avoid confusion, we use the notations $\inte,\cl$ and $\partial$  only with respect to the graph structure of the modified graph $(V,E)$.
Set $B_n^o=B_{d_{\sigma}}(\bar{x},R_n-1)\cap V_o$ and $E_n^o=E_o\vert_{B_n^o\times B_n^o}$.
Recall that $V_e$ is the set of new points added for an edge $e\in E_o$
(see Definition \ref{defi-modified-graph} (1) above).
Let $B_n'=B_n^o\cup \bigcup_{e\in E_n^o }V_e$.
We define
$B_n=\inte B_n'.$ 
\begin{lem}\label{lem-Bn}
The following assertions hold.
\begin{enumerate}
    \item $\partial B_n\subseteq V_o$;
    \item $B_{d_{\sigma}}(\bar{x},R_n-3)\subseteq B_n\subseteq B_{d_{\sigma}}(\bar{x},R_n)$;
    \item $\forall n\ge 1$, $x\in \partial B_{n-1}, y\in \partial B_{n}\Rightarrow d_{\sigma}(x,y)\ge R_{n-1}-3$.
\end{enumerate}
\end{lem}

\begin{proof}
By the construction of $(V, E)$, we see that $B_n\cap V_o$ is simply the interior $\tilde{B}_n^o$ of $ B_n^o$
with respect to the original graph structure $(V_o, E_o)$.

\noindent
(1) \ Note that $\bigcup_{e\in E_n^o }V_e\subseteq B_n$ since $\cl(\bigcup_{e\in E_n^o }V_e)\subseteq B_n'$. We also have $\partial(\bigcup_{e\in E_n^o }V_e)\subseteq B_n^o$.
So $B_n$ is characterized as $B_n=\tilde{B}_n^o\cup \bigcup_{e\in E_n^o }V_e$. The assertion (1) follows.

\noindent
(2) \
Similarly we have $B_{d_{\sigma}}(\bar{x},R_n-2)\cap V_o\subseteq B_n$, since $\cl(B_{d_{\sigma}}(\bar{x},R_n-2)\cap V_o)\subseteq B_n'$.
For $x\in B_{d_{\sigma}}(\bar{x},R_n-3)\cap V_o^c$, we have $x\in \bigcup_{e\in E_n^o }V_e$. It follows that $B_{d_{\sigma}}(\bar{x},R_n-3)\subseteq B_n$.
The assertion (2) holds by observing
that  $B_n\subseteq B_n'\subseteq B_{d_{\sigma}}(\bar{x},R_n-1+1/2)$.

\noindent
(3) \ For any $n\ge 1$ and $x\in \partial B_{n}$, we have shown that $x\in B_{d_{\sigma}}(\bar{x},R_n-3)^c\cap B_{d_{\sigma}}(\bar{x},R_n)$.
The assertion (3) follows by noting $R_n-R_{n-1}=R_{n-1}$.
\end{proof}

\begin{rem}\rm
By Lemma \ref{lem-Bn}, we can  regard $\{B_n\}_{n\in \NN}$  as an increasing sequence of ``balls" in $V$
so that all boundary points are ``forced'' in $V_o$.
\end{rem}\bigskip


\subsection{Estimate for the heat equation}
The key technical problem is to
estimate the quantity
\[\PP_{\bar{x}}(\tau_n-\tau_{n-1}\leq c_n).\]
By the strong Markov property of the Markov chain $\xt$,
we have
\begin{equation}\label{eq-markov-prop}
  \PP_{\bar{x}}(\tau_n-\tau_{n-1}\leq c_n)
=\mathbb{E}_{\bar{x}}(\PP_{X_{\tau_{n-1}}}(\tau_n\leq c_n)).
\end{equation}
Set \[r_n= 2^{n+2}-3\ge 5,\]
for $n\ge 1$.
Then since
$$
r_n<R_{n}-R_{n-1}-3=R_{n-1}-3
$$
for all $n\ge 1$, we see by (3) of Lemma \ref{lem-Bn} that
the Markov chain $\lp X_t\rp_{t\ge 0}$ must run out of a ball $B_{d_\sigma} (X_{\tau_{n-1}}, r_n)$
before it leaves $B_n$. 
Moreover, by noting that
\[X_{\tau_{n-1}}\in \partial B_{n-1}, X_{\tau_{n}}\in \partial B_{n},\]
it follows  from (\ref{eq-markov-prop}) that
\[\PP_{\bar{x}}(\tau_n-\tau_{n-1}\leq c_n)\leq \sup_{z\in\partial B_{n-1}}\PP_{z}(\tau_{B_{d_{\sigma}}(z,r_n)}\leq
  c_n).\]

For a fixed $z\in\partial B_{n-1}$, define
\[u(x,t)=\PP_x (\tau_{B_{d_{\sigma}}(z,r_n)}\leq
  t)\]
as a function on $B_{d_{\sigma}}(z,r_n)\times[0,\infty)$.
We now estimate $u(x,t)$ by using
Propositions \ref{prop-heat-solution} and \ref{prop-integral-maximum-principle}.
In order to do so, we first show the following lemma.
Set $K_n=B_{d_{\sigma}}(z,r_n-\sigma_n)$ and $L_n=B_{d_{\sigma}}(z,r_n)$.

\begin{lem}\label{k-l}
{\rm (i)}  \ For any $x\in L_n$ and $y\in V$ with $x\sim y$, the inequality
$d_{\sigma}(x,y)\le \sigma_n$ holds.

\noindent
{\rm (ii)}\ The inclusion $\cl(K_n)\subseteq L_n$ holds.
\end{lem}
\begin{proof}
(i) \
For any  $x\in L_n$ 
we have by the triangle inequality,
\[d_{\sigma}(x, \bar{x})
\ge R_{n-1}-3-r_n=2^{n+2},\]
and thus $d_{\sigma}(x,y)\le\sigma(x,y)\le \sigma_n$ for $y\sim x$ by Lemma \ref{lem-shrinking}.
\noindent

(ii) \ Since
$\sigma(x,y)\leq \sigma_n$
for any $x\in K_n, y\in V$ with $x\sim y$, 
we obtain
$\cl(K_n)\subseteq L_n$.

\end{proof}

By Lemma \ref{k-l} (ii) and Proposition \ref{prop-heat-solution},
the function $u(x,t)$ is a solution to
the heat equation
\begin{equation*}
\frac{\partial}{\partial t}u(x,t)+\Delta u(x,t)=0
\end{equation*}
on $K_n\times [0, \infty)$ with $u(x,0)\equiv 0$. 
In order to apply Proposition \ref{prop-integral-maximum-principle} to $u(x,t)$ for $K=K_n$ and $L=L_n$,
we choose
the auxiliary functions to be
\begin{equation}\label{def-aux}
\xi(x,t)=-2\a_{n}^2 e^{4} t-2\a_{n} d_{\s}(x,z)
\text{~~and~~}\eta(x)=\frac{(e^{\a_{n}(r_n-\sigma_n)}-e^{\a_{n} d_{\s}(x,z)})_+}{e^{\a_{n}(r_n-\sigma_n)}-1}
\end{equation}
with
$$\a_{n}=\frac{2f(R_n)+2\log\log R_n}{r_n}.$$
The key is that
\begin{equation}\label{ineq-key}
\a_n\sigma_n\le \frac{2}{r_n}\le 1.
\end{equation}
Using this inequality, we next obtain the following lemma.
\begin{lem}
The following inequality holds for any $n\geq 1$.
$$\sum_{x\in K_n}u^2(x,s)\eta^2(x)e^{\xi(x,s)}\mu(x)
\leq 2\int_{0}^{s}\sum_{x\in L_n}\sum_{y\in
L_n}w(x,y)(\eta(x)-\eta(y))^2 e^{\xi(x,t)}dt.$$
\end{lem}
\begin{proof}
We first verify that
the conditions (1)-(4) in Proposition \ref{prop-integral-maximum-principle}
hold for $K=K_n$ and $L=L_n$,  $\xi$ and  $\eta$ as above.
We now check the condition (4) in Proposition \ref{prop-integral-maximum-principle}.
It follows by the triangle inequality that
$$|\xi(y,t)-\xi(x,t)|\leq 2\a_n|d_{\sigma}(x,z)-d_{\sigma}(y,z)|\leq 2\a_nd_{\sigma}(x,y).$$
Hence by  the inequality
$|e^t-1|\leq |t|e^t \ (t\in\RR)$, we have
\begin{equation}\label{ineq-aux}
(1-e^{\xi(y,t)-\xi(x,t)})^2\leq 4\alpha_n^2d_{\sigma}(x,y)^2e^{4\alpha_nd_{\sigma}(x,y)}.
\end{equation}
Then for any $x,y\in L_n$ with $x\sim y$,  we see by Lemma \ref{k-l} (i) and(\ref{ineq-key}) that
$$\a_n d_{\sigma}(x,y)\leq \a_n\sigma_n\le 1.$$
We also know that $d_{\sigma}(x,y)\leq \sigma(x,y)$ for any $x,y\in V$ with $x\sim y$
by the definition of the adapted metric.
Therefore, by these inequalities, the right hand side of (\ref{ineq-aux}) above is less than
$4\a_n^2\sigma(x,y)^2e^{4}$
for any $x,y\in L_n$ with $x\sim y$.
We further recall that the weighted function $\sigma$ is adapted, that is,
\begin{equation}\label{adapted-again}
\sum_{y\in L_n}w(x,y)\sigma(x,y)^2\leq 1
\end{equation}
for any $x\in L_n$.
As a consequence, we have for any $x\in L_n$ and $t>0$,
\begin{align*}
&\mu(x)\frac{\partial}{\partial t}\xi(x,t)+\frac{1}{2}\sum_{y\in L_n}w(x,y)(1-e^{\xi(y,t)-\xi(x,t)})^2\\
\le&-2\a_{n}^2 e^{4}\mu(x)+\frac{1}{2}\sum_{y\in L_n}w(x,y)\cdot 4\a_n^2\sigma^2(x,y)e^4\le 0.
\end{align*}
Namely, the condition (4) holds. We can also check the conditions (1)-(3) directly.
We finally see by  Proposition \ref{prop-integral-maximum-principle} that
 \begin{eqnarray}
   \begin{aligned}
&\sum_{x\in
K_n}u^2(x,s)\eta^2(x)e^{\xi(x,s)}\mu(x)\\
\leq &2\int_{0}^{s}\sum_{x\in L_n}\sum_{y\in
L_n}w(x,y)(\eta(x)-\eta(y))^2 u(y,t)^2e^{\xi(x,t)}dt\\
\leq &2\int_{0}^{s}\sum_{x\in L_n}\sum_{y\in
L_n}w(x,y)(\eta(x)-\eta(y))^2 e^{\xi(x,t)}dt,
   \end{aligned}
 \end{eqnarray}
where we used the fact that $0\leq u\leq 1$.
\end{proof}

We finally estimate the quantity $\PP_{\bar{x}}(\tau_n-\tau_{n-1}\leq c_n)$ as follows.
\begin{prop}
For any $s>0$, the following inequality holds.
\begin{equation}\label{equation-exit-time}
  \PP_{\bar{x}}(\tau_n-\tau_{n-1}\leq s)\le 2\exp\left\{e^{4}\a_n^2s-\frac{3}{2}f(R_n)-2\log\log R_n\right\}.
\end{equation}
\end{prop}
\begin{proof}
In a similar way to deriving (\ref{ineq-aux}), we obtain
\begin{equation*}
\begin{split}
\left(e^{\a_{n} d_{\sigma}(x,z)}-e^{\a_{n} d_{\sigma}(y,z)}\right)^2
&=e^{2\a_n d_{\sigma}(x,z)}\left(1-e^{\a_n(d_{\sigma}(y,z)-d_{\sigma}(x,z))}\right)^2\\
&\leq \a_{n}^2e^{2\a_nd_{\sigma}(x,z)+2\a_n\sigma_n}\sigma(x,y)^2.
\end{split}
\end{equation*}
Then by (\ref{ineq-key}) and (\ref{adapted-again}),
we have the following estimate for any $x\in L_n$,
\begin{eqnarray}
\begin{aligned}
&\sum_{y\in
L_n}w(x,y)(\eta(x)-\eta(y))^2\\
\leq &\frac{1}{(e^{\a_{n}(r_n -\sigma_n)}-1)^2}\sum_{y\in
L_n}w(x,y)\left(e^{\a_{n} d_{\s}(x,z)}-e^{\a_{n} d_{\s}(y,z)}\right)^2\\
\leq &\frac{\a_{n}^2 e^{2\a_{n} d_{\s}(x,z)+2\a_{n}\sigma_n}}{(e^{\a_{n}(r_n
-\sigma_n)}-1)^2}\sum_{y\in V}w(x,y)\sigma^2(x,y)\\
\leq &\frac{\a_{n}^2 e^{2\a_{n} d_{\s}(x,z)+2}}{(e^{\a_{n}(r_n-\sigma_n)}-1)^2}\mu(x).
\end{aligned}
\end{eqnarray}
Therefore, it follows that for any $z\in \partial B_{n-1}$
\begin{align*}
&u^2(z,s)\mu(z)e^{-2\a_{n}^2 e^{4}s}\\
\le &\sum_{x\in
K_n}u^2(x,s)\eta^2(x)e^{\xi(x,s)}\mu(x)\\
\leq&2\int_{0}^{s}\sum_{x\in L_n}\sum_{y\in
L_n}w(x,y)(\eta(x)-\eta(y))^2 e^{\xi(x,t)}dt\\
\leq&2\int_{0}^{s}\sum_{x\in L_n}\frac{\a_{n}^2 e^{2\a_{n}
d_{\s}(x,z)+2}}{(e^{\a_{n}(r_n-\sigma_n)}-1)^2} \times e^{-2\a_{n} d_{\s}(x,z)-2\a_{n}^2 e^{4}t}\mu(x)dt\\
= &\frac{e^{-2}(1-\exp\lp -2\a_{n}^2
e^{4}s\rp)}{(e^{\a_{n}(r_n-\sigma_n)}-1)^2}\mu(L_n).
\end{align*}
Here we recall
that $\partial B_{n-1}\subseteq V_o$, which is the key point of construction of $(B_n)_{n\in\NN}$.
By assumption, we have
$\mu(z)\geq C_o>0$.
Hence it follows that
\begin{equation*}
u^2(z,s)\leq\frac{1}{C_o}\frac{e^{-2}(e^{2\a_{n}^2
e^{4}s}-1)}{(e^{\a_{n}(r_n-\sigma_n)}-1)^2}\mu(B_{d_{\s}}(z,r_n)).
\end{equation*}
In particular,
by noting that $B_{d_{\s}}(z,r_n)\subset B_{d_{\s}}(\overline{x}, R_n)$,
we obtain
\begin{align}\label{eq-main-estimate}\nonumber
\PP_{\bar{x}}(\tau_n-\tau_{n-1}\leq s)&\leq \sup_{z\in\partial B_{n-1}}u(z, s)
\leq\sqrt{\frac{\mu(B_{d_{\sigma}}(\bar{x},R_n))}{e^2C_o}}
\frac{e^{\a_{n}^2 e^{4}s}}{e^{\a_{n}(r_n-\sigma_n)}-1}\\
&\le \sqrt{\frac{3\mu_o(B_{d_{\sigma_o}}(\bar{x},R_n))}{e^2C_o}}
\frac{e^{\a_{n}^2 e^{4}s}}{e^{\a_{n}(r_n-\sigma_n)}-1} \le \sqrt{\frac{3\exp f(R_n)}{e^2}}
\frac{e^{\a_{n}^2 e^{4}s}}{e^{\a_{n}(r_n-\sigma_n)}-1},
\end{align}
where in the second last inequality we applied (2) of Lemma \ref{lem-geometry}.

Note
the elementary fact that if $r_n\ge2$, then \[\frac{1}{2}\a_n
r_n\le \a_{n}(r_n-\sigma_n)\]
by (\ref{ineq-key}). This implies that
\begin{equation}\label{ineq-ar}
1-e^{-\a_{n}(r_n-\sigma_n)}\ge 1-e^{-\frac{1}{2}\a_n r_n}\ge \frac{\a_n r_n}{\a_n r_n+2}.
\end{equation}
Since $R_n=2^{n+4}$ and
\begin{equation}\label{eq-ar}
\a_n r_n=2(f(R_n)+\log\log R_n),
\end{equation}
it follows that $\a_n r_n\ge 2\log\log16\ge 2$.
Together with (\ref{ineq-ar}) and (\ref{eq-ar}),
substituted for the inequality (\ref{eq-main-estimate}), we have
\begin{align*}
    \PP_{\bar{x}}(\tau_n-\tau_{n-1}\leq s)&\leq
    \sqrt{3}\left(1+\frac{2}{\a_n r_n}\right)
    \exp\left\{\a_n^2e^{4}s+\frac{1}{2}f(R_n)
    -\a_n r_n\right\}\\
    &\le 2\sqrt{3}\exp\left\{\a_n^2e^{4}s+\frac{1}{2}f(R_n)-\a_n r_n\right\} \\
&= 2\sqrt{3}\exp\left\{e^{4}\a_n^2s-\frac{3}{2}f(R_n)-2\log\log R_n\right\},
\end{align*}
which completes the proof.
\end{proof}

We are now in a position to finish the proof of 
Theorem \ref{thm-main} by choosing proper $c_n$.
Recall that $r_n=2^{n+2}-3\ge \frac{1}{16}R_n$. Choose \[c_n=\frac{ r_n^2 }{8e^4(f(R_n)+\log\log R_n)}\]
so that  \[\a_n^2c_n=\frac{1}{2e^4}(f(R_n)+\log\log R_n).\]
By (\ref{equation-exit-time}), we have
\begin{align*}
  \PP_{\bar{x}}(\tau_n-\tau_{n-1}\leq c_n)&\le 2\sqrt{3}\exp\left\{\frac{1}{2}(f(R_n)+\log\log R_n)-\frac{3}{2}f(R_n)-2\log\log R_n\right\}\\
  &\le 2\sqrt{3}(\log R_n)^{-3/2},
\end{align*}
and thus
\[\sum_{n=1}^{\infty}\PP_{\bar{x}}(\tau_n-\tau_{n-1}\leq c_n)<\infty.\]

We are left to verify that $\sum_{n=1}^{\infty} c_n=\infty$. Indeed, we have
\begin{align*}
    t_{m}= \sum_{n=1}^{m}c_n
    &= \sum_{n=1}^{m}\frac{ r_n^2 }{8e^4(f(R_n)+\log\log R_n)}\\
    &\ge\sum_{n=1}^{m}\frac{ R_n^2 }{2048e^4(f(R_n)+\log\log R_n)}\\
    &= \frac{1}{4096e^4}\sum_{n=1}^{m}\frac{R_{n+1}(R_{n+1}-R_{n})}{f(R_n)+\log\log R_n}\\
&\ge \frac{1}{4096e^4}\int_{R_1}^{R_{m+1}}\frac{rdr}{f(r)+\log\log r},
\end{align*}
which approaches to $\infty$ by the volume growth assumption.
We now define  $\psi$ by
  \[
  \psi(R)=\frac{1}{8192e^4}\int_{32}^{R}\frac{rdr}{f(r)+\log\log r}.\]
Since $t_m-\psi(R_{m+1})\rightarrow \infty$ as $m\rightarrow \infty$ by assumption,
$\psi^{-1}(t)$ is an upper rate function for $\g$ by Lemma \ref{lem-SC-modified-graph}.
Combining this with  Corollary \ref{cor-local-time}, we have Theorem \ref{thm-main} for the original weighted graph $\go$.


\section{Examples}\label{sect-examples}

We apply Theorem \ref{thm-main} and Corollary \ref{cor-main} to several weighted graphs.

\begin{eg}\label{bd-process} \rm (Birth and death processes, \cite[Examples 1.19 and 1.20]{Huang-escape}) \
Let ${\mathbb Z}_+$ be the totality of nonnegative integers and
consider the weighted graph $({\mathbb Z}_+, w, \mu)$.
We assume that $w$ is a symmetric weight function on ${\mathbb Z}_+\times {\mathbb Z}_+$ such that
for $(x,y)\in {\mathbb Z}_+\times {\mathbb Z}_+$,
$$w(x,y)>0 \iff |x-y|=1.$$
For simplicity, we also assume that
\begin{equation}\label{as-adapted}
\mu(n)\leq 2w(n,n+1)\quad \hbox{for any $n\in {\mathbb Z}_+$}.
\end{equation}
Define the weight function $\sigma$ by
$$\sigma(n,n+1)=\sqrt{\frac{\mu(n)}{2w(n,n+1)}}.$$
Then $\sigma$ is adapted by definition and (\ref{as-adapted}), and the adapted path metric $d_{\sigma}$ is given by
$$d_{\sigma}(m,n)=\sum_{k=m}^{{n-1}}\sigma(k,k+1) \quad \hbox{for $n>m \ (m,n\in {\mathbb Z}_+)$.}$$

In the sequel, we further assume that $\mu\asymp 1$ and
$$w(n,n+1)\asymp \frac{1}{2}(n+1)^2\log(n+2)^{\beta}\log\log(n+3)^{\gamma}, \quad n\in {\mathbb Z}_+$$
for some $\beta,\gamma\geq 0$ with (\ref{as-adapted}). Since
$$\sigma(n,n+1)\asymp\frac{1}{(n+1)\log(n+2)^{\beta/2}\log\log(n+3)^{\gamma/2}},$$
we obtain for $0\leq \beta<2$ or $\beta=2$ with $0\leq \gamma\leq 2$,
$$d_{\sigma}(0,n)\asymp
\begin{cases}
\displaystyle \frac{\log(n+2)^{1-\beta/2}}{\log\log(n+3)^{\gamma/2}} & 0\leq \beta<2\\
\displaystyle \log \log(n+3)^{1-\gamma/2} & \beta=2, \  0\leq \gamma<2\\
\displaystyle \log\log\log(n+30) & \beta=\gamma=2.
\end{cases}$$
Hence  for all large $r>0$,
$$
\log\mu(B_{d_{\sigma}}(0,r))\asymp
\begin{cases}
\displaystyle r^{2/(2-\beta)}(\log r)^{\gamma/(2-\beta)} & 0\leq \beta<2\\
\displaystyle \exp\left(cr^{2/(2-\gamma)}\right) &  \beta=2, \ 0\leq \gamma<2\\
\displaystyle \exp(\exp(cr)) & \beta=\gamma=2.
\end{cases}$$
On the other hand, if  $\beta>2$ or $\beta=2$ with $\gamma>2$, then $\sup_{n\in\ZZ_+}d_{\sigma}(0,n)<\infty$ and
thus $\mu(B_{d_{{\s}}}(0,r))=\infty$ for all large $r>0$.

If $0\leq \beta<1$ or $\beta=1$ with $0\leq \gamma\leq 1$,
then the corresponding birth and death process is conservative by Theorem \ref{thm-main}.
Moreover, if we denote by $\phi(t)$ an upper rate function for the process,
then in a similar way to Corollary \ref{cor-main}, we have
$$
\phi(t)\asymp
\begin{cases}
\displaystyle t^{(2-\beta)/(2-2\beta)}(\log t)^{\gamma/(2-2\beta)} & 0\leq \beta<1\\
\displaystyle \exp\left(ct^{1/(1-\gamma)}\right) &  \beta=1, \ 0\leq \gamma< 1\\
\displaystyle \exp(\exp(ct)) &  \beta=1, \ \gamma=1.\\
\end{cases}$$
\end{eg}


\begin{eg}\label{exam-anti-tree}\rm (Anti-trees, \cite[Example 4.7]{Huang-escape} and \cite{WOJ-Survey}) \
Let $\{S_n\}_{n\in \NN}$ be a sequence of disjoint, finite and non-empty sets such that
$S_0=\{x_0\}$. We set
$$V=\bigcup_{n\in \NN} S_n.$$
Let $\rho_0$ be a function on $V$ such that $\rho_0(x)=n$ for $x\in S_n$. We define
$$E=\left\{(x,y)\in V\times V : |\rho_0(x)-\rho_0(y)|=1\right\}.$$
Then  the pair of $V$ and $E$ form a symmetric graph $(V,E)$ such that
every vertex in $S_n$ is connected to every vertex in $S_{n+1}$
and the associated graph metric $\rho$ satisfies   $\rho_0(x)=\rho(x_0,x)$ for any $x\in V$.

We assume that each vertex $x\in V$ has $c(x)[(\rho_0(x)+2)^{\alpha}(\log (\rho_0(x)+3))^{\beta}]$ neighbors
of vertices with graph distance $\rho_0(x)+1$, where $c(x)$ is a function on $V$
such that $c\asymp 1$.
We  also assume that
$\mu\asymp 1$ and $w(x,y)\asymp {\bf 1}_{\{(x,y)\in E\}}$.
Let ${\rm Deg}(x)$ be a weighed degree of the vertex $x\in V$ defined by
$${\rm Deg}(x)=\frac{1}{\mu(x)}\sum_{y\in V, \,  x\sim y}w(x,y).$$
and
$$\sigma(x,y):=\frac{1}{\sqrt{{\rm Deg}(x)}}\wedge \frac{1}{\sqrt{{\rm Deg}(y)}}\wedge 1, \quad x,y \in V, x\sim y.$$
Then $\sigma$ is a weight function adapted to the weighted graph $(V, w, \mu)$.
Since
$$\sigma(x,y)\asymp \frac{1}{(\rho_0(x)+2)^{\alpha/2}(\log(\rho_0(x)+3))^{\beta/2}}$$
for any $(x,y)\in E$,
the adapted path metric $d_{\sigma}$ satisfies for $0\leq \alpha<2$ or $\alpha=2$ with $0\leq \beta\leq 2$,
$$d_{\sigma}(x_{0},x)\asymp
\begin{cases}
\displaystyle \frac{(\rho_0(x)+2)^{1-\alpha/2}}{(\log(\rho_0(x)+3))^{\beta/2}} & 0\leq \alpha<2\\
\displaystyle (\log(\rho_0(x)+3))^{1-\beta/2} & \alpha=2, \  0\leq \beta<2\\
\displaystyle \log\log(\rho_0(x)+3) & \alpha=\beta=2.
\end{cases}$$
Therefore, we obtain for all large $r>0$,
$$\log \mu(B_{d_{\sigma}}(x_{0},r))\asymp
\begin{cases}
\displaystyle  \log r & 0\leq \alpha<2\\
\displaystyle r^{2/(2-\beta)} &  \alpha=2, \ 0\leq \beta<2\\
\displaystyle e^{cr} & \alpha=\beta=2.
\end{cases}$$
On the other hand, if  $\alpha>2$ or $\alpha=2$ with $\beta>2$, then $\sup_{x\in V}d_{\sigma}(x,x)<\infty$ and
thus $\mu(B_{d_{{\s}}}(x,r))=\infty$ for all large $r>0$.

If $0\leq \alpha<2$ or $\alpha=2$ with $0\leq \beta\leq 1$,
then the corresponding birth and death process is conservative by Theorem \ref{thm-main}.
Furthermore, if we denote by $\phi(t)$ an upper rate function for the process,
then Corollary \ref{cor-main} implies that
$$
\phi(t)\asymp
\begin{cases}
\displaystyle \sqrt{t\log t} &  0\leq \alpha<2, \\
\displaystyle  t^{(2-\beta)/(2-2\beta)} & \alpha=2, \ 0\leq \beta< 1,\\
\displaystyle e^{ct} & \alpha=2, \ \beta=1.
\end{cases}$$
\end{eg}

\begin{eg}\label{exam-tree}\rm (Trees, \cite[Example 4.6]{Huang-escape}) \
For $\alpha\geq 0$ and $\beta\geq 0$, let $T_{\alpha,\beta}=(V, E)$ be a tree rooted at $x_{0}$
and $\rho$ the associated graph distance.
For $x\in V$, let $\rho_0(x)$ be the distance between $x_{0}$ and $x$,
that is, $\rho_0(x)=\rho(x_{0}, x)$ for $x\in V$.

As in Example \ref{exam-anti-tree}, we assume that each vertex $x\in V$ has $c(x)[(\rho_0(x)+2)^{\alpha}(\log (\rho_0(x)+3))^{\beta}]$ neighbors
of vertices with graph distance $\rho_0(x)+1$, where $c(x)$ is a function on $V$
bounded from below and above by positive constants.
We also assume that
$\mu\asymp 1$ and $w(x,y)\asymp {\bf 1}_{\{(x,y)\in E\}}$.
Let ${\rm Deg}(x)$ be a weighed degree of the vertex $x\in V$ defined by
$${\rm Deg}(x)=\frac{1}{\mu(x)}\sum_{y\in V, \,  x\sim y}w(x,y).$$
and
$$\sigma(x,y):=\frac{1}{\sqrt{{\rm Deg}(x)}}\wedge \frac{1}{\sqrt{{\rm Deg}(y)}}\wedge 1, \quad x,y \in V, x\sim y.$$
Then $\sigma$ is a weight function adapted to the weighted graph $(V, w, \mu)$.
Since
$$\sigma(x,y)\asymp \frac{1}{(\rho_0(x)+2)^{\alpha/2}(\log(\rho_0(x)+3))^{\beta/2}}$$
for any $(x,y)\in E$,
the adapted path metric $d_{\sigma}$ satisfies for  $0\leq \alpha<2$ or $\alpha=2$ with $0\leq \beta\leq 2$,
$$d_{\sigma}(x_{0},x)\asymp
\begin{cases}
\displaystyle \frac{(\rho_0(x)+2)^{1-\alpha/2}}{(\log(\rho_0(x)+3))^{\beta/2}} & 0\leq \alpha<2\\
\displaystyle (\log(\rho_0(x)+3))^{1-\beta/2} & \alpha=2, \  0\leq \beta<2\\
\displaystyle \log\log(\rho_0(x)+3) & \alpha=\beta=2.
\end{cases}$$
Therefore, we obtain for all large $r>0$,
$$\log \mu(B_{d_{\sigma}}(x_{0},r))\asymp
\begin{cases}
\displaystyle r^{2/(2-\alpha)}(\log r)^{1+\beta/(2-\alpha)} & 0\leq \alpha<2\\
\displaystyle \exp(cr^{2/(2-\beta)}) &  \alpha=2, \ 0\leq \beta<2\\
\displaystyle {\exp}(\exp(cr)) & \alpha=\beta=2.
\end{cases}$$
On the other hand, if  $\alpha>2$ or $\alpha=2$ with $\beta>2$, then $\sup_{x\in V}d_{\sigma}(x,x)<\infty$ and
thus $\mu(B_{d_{{\s}}}(x,r))=\infty$ for all large $r>0$.

If $0\leq \alpha<1$ or $\alpha=1$ with $\beta=0$,
then the corresponding birth and death process is conservative by Theorem \ref{thm-main}.
Furthermore, if we denote by $\phi(t)$ an upper rate function for the process,
then in a similar way to Corollary \ref{cor-main}, we find that
$$
\phi(t)\asymp
\begin{cases}
\displaystyle t^{(2-\alpha)/(2-2\alpha)}(\log t)^{(2-\alpha+\beta)/(2-2\alpha)} &  0\leq \alpha<1, \\
\displaystyle e^{ct} & \alpha=1, \ \beta=0.
\end{cases}$$
\end{eg}

\begin{rem}\rm
(i) \ Theorem \ref{thm-main} is sharp for Examples \ref{bd-process} and \ref{exam-anti-tree}.
For the conservativeness, Folz \cite[Examples 1 and 2]{FolzSC} mentioned the sharpness
by using the results of Wojciechowski \cite{WOJ-Survey}.
For the escape rate, the sharpness is checked in \cite[Example 1.20 and Section 6]{Huang-escape}
by using an alternative approach (\cite[Theorem 1.15]{Huang-escape}).

On the contrary, Theorem \ref{thm-main} is not sharp for Example \ref{exam-tree}.
In fact, if we take $\alpha=1$, then it follows  by Wojciechowski \cite[Remark 4.3]{WOJ-Survey} that
the associated process is conservative if and only if $0\leq \beta\leq 1$.
Moreover, for $0\leq \alpha<1$ with $\beta=0$, we know by (6.14) of \cite{Huang-escape}
that the function $\phi(t)\asymp t^{(2-\alpha)/(2-2\alpha)}$ can be an upper rate function.

\noindent
(ii) \ Even though the process is not conservative, we have the following information from the adapted metric:
if any ball associated with the adapted metric is relatively compact,
then the Silverstein extension of the corresponding Dirichlet form is uniquely determined
(see \cite{Ku-Sh}).
Roughly speaking, this means that we can not extend the process after the lifetime.
\end{rem}

\begin{eg}\rm
We consider a random walk on ${\mathbb Z}^d$.
We assume that $\mu(x)=1$ for any $x\in {\mathbb Z}^d$
and the weight function $w(x,y)$ satisfies for some $\alpha\geq 0$ and $\beta\geq 0$,
$$w(x,y)\asymp
\{(|x|+2)^{\alpha}(\log(|x|+3))^{\beta}+(|y|+2)^{\alpha}(\log(|y|+3))^{\beta}\}{\bf 1}_{\{|x-y|=1\}},$$
where $|\cdot|$ is the graph distance on ${\mathbb Z}^d$. Then
$$w(x,y)\asymp (|x|+2)^{\alpha}(\log(|x|+3))^{\beta}{\bf 1}_{\{|x-y|=1\}}.$$

Let ${\rm Deg}(x)$ be a weighed degree of the vertex $x\in V$ defined by
$${\rm Deg}(x)=\frac{1}{\mu(x)}\sum_{y\in {\mathbb Z}^d, \, |x-y|=1}w(x,y)$$
and
$$\sigma(x,y):=\frac{1}{\sqrt{{\rm Deg}(x)}}\wedge \frac{1}{\sqrt{{\rm Deg}(y)}}\wedge 1
\quad \hbox{for any $x,y \in {\mathbb Z}^d$ with $|x-y|=1$.}$$
Then $\sigma$ is a weight function adapted to the weighted graph $({\mathbb Z}^d, w, \mu)$.
Since
$$\sigma(x,y)\asymp \frac{1}{(|x|+2)^{\alpha/2}(\log(|x|+3))^{\beta/2}}$$
for any $x,y\in {\mathbb Z}^d$ with $|x-y|=1$,
the adapted path metric $d_{\sigma}$ satisfies for any $0\leq \alpha<2$ or $\alpha=2$ with $0\leq \beta\leq 2$,
$$d_{\sigma}(0,x)\asymp
\begin{cases}
\displaystyle \frac{(|x|+2)^{1-\alpha/2}}{(\log(|x|+3))^{\beta/2}} & 0\leq \alpha<2\\
\displaystyle (\log(|x|+3))^{1-\beta/2} & \alpha=2, \  0\leq \beta<2\\
\displaystyle \log\log(|x|+3) & \alpha=\beta=2.
\end{cases}$$
We thus obtain for all large $r>0$,
$$\log\mu(B_{d_{\sigma}}(x,r))\asymp
\begin{cases}
\displaystyle \log r& 0\leq \alpha<2\\
\displaystyle r^{2/(2-\beta)} &  \alpha=2, \ 0\leq \beta<2\\
\displaystyle e^{cr} & \alpha=\beta=2.
\end{cases}$$
On the other hand, if  $\alpha>2$ or $\alpha=2$ with $\beta>2$, then $\sup_{x\in \ZZ^d}d_{\sigma}(0,x)<\infty$ and
thus $\mu(B_{d_{{\s}}}(0,r))=\infty$ for all large $r>0$.

It follows by Theorem \ref{thm-main} that
the associated random walk is conservative if $0\leq \alpha<2$ or $\alpha=2$ with $0\leq \beta\leq 1$.
Furthermore, we have
$$
\phi(t)\asymp
\begin{cases}
\displaystyle \sqrt{t\log t} & 0\leq \alpha<2\\
\displaystyle t^{(2-\beta)/(2-2\beta)} &  \alpha=2, \ 0\leq \beta<1\\
\displaystyle e^{ct} & \alpha=2, \ \beta=1.
\end{cases}$$
\end{eg}

\section*{Acknowledgement}
XH would like to thank Prof. Grigor'yan for inspiring discussions.
YS was supported in part by the Grant-in-Aid for Young Scientists (B) No.~23740078. We are grateful to M. Schmidt for a careful reading of a preliminary version of this article and helpful comments.
\bibliographystyle{amsplain}
\bibliography{ref}

\end{document}